\documentclass[12pt, reqno]{amsart}

\usepackage{a4wide}
\usepackage{esint,amsfonts,amsmath,amssymb,epsfig,stmaryrd,mathrsfs,bm,accents,
mathtools,color,bm,epic,ifthen,eepic,
caption}
\begingroup

\newtheorem{theorem}{Theorem}[section]
\newtheorem{lemma}[theorem]{Lemma}
\newtheorem{propos}[theorem]{Proposition}
\newtheorem{corol}[theorem]{Corollary}

\endgroup

\theoremstyle{definition}
\newtheorem{definition}[theorem]{Definition}

\newtheorem{remark}[theorem]{Remark}

\numberwithin{equation}{section}

\newcommand\proj{\mathscr{P}}
\newcommand{\eps}{{\varepsilon}}

\newcommand{\mint}{\hspace{0.1cm}-\hspace{-0.45cm}\int}

\newcommand{\Lip}{{\rm {Lip}}}

\newcommand{\dist}{{\rm {dist}}}

\newcommand{\dv}{{\text {div}}}

\newcommand\weak{{\rightharpoonup}\,}

\newcommand\supp{{\rm supp}\,}
\newcommand\res{\mathop{\hbox{\vrule height 7pt width .3pt depth 0pt
\vrule height .3pt width 5pt depth 0pt}}\nolimits}
\newcommand\Id{{\rm Id}\,}

\def\R#1{{\mathbb R}^{#1}}
\newcommand\Z{{\mathbb Z}}

\newcommand\N{{\mathbb N}}

\newcommand{\mass}{{\mathbf{M}}}

\def\a#1{\left\llbracket{#1}\right\rrbracket}

\newcommand{\norm}[2]{\left\|#1\right\|_{#2}}
\newcommand{\ra}{\right\rangle}
\newcommand{\la}{\left\langle}

\newcommand{\de}{\partial}

\newcommand{\ph}{\varphi}

\newcommand{\n}{\notag}
\newcommand{\E}{\textup{Ex}}

\renewcommand{\l}{\left}
\renewcommand{\r}{\right}
\newcommand{\cyl}{{\mathcal{C}}}

\newcommand{\gr}{\textup{graph}}

\begin{document}

\title{Center manifold: a case study}

\author[C.~De Lellis]{Camillo De Lellis}
\address{Z\"urich University}
\email{camillo.delellis@math.uzh.ch}
\author[E.~N.~Spadaro]{Emanuele Nunzio Spadaro}
\address{Hausdorff Center for Mathematics, Universit\"at Bonn}
\email{emanuele.spadaro@hcm.uni-bonn.de}

\maketitle

\begin{abstract}
Following Almgren's construction of the \textit{center manifold} in
his Big regularity paper, we show the $C^{3,\alpha}$ regularity of area-minimizing
currents in the neighborhood of points of density one
without using the nonparametric theory.
This study is intended as a first step towards the understanding of Almgren's 
construction in its full generality.
\end{abstract}

\section{Introduction}
In this note we consider area-minimizing integral currents $T$ of dimension $m$ in $\R{m+n}$.
The following theorem is the cornerstone of the regularity theory.
It was proved for the first time by De Giorgi \cite{DG} for $n=1$ and then extended later by several authors (the constant $\omega_m$ denotes, as usual, the Lebesgue measure
of the $m$-dimensional unit ball).

\begin{theorem}\label{t:degiorgi}
There exist constants $\eps,\beta>0$ such that, if $T$ is an area-minimizing integral current
and $p$ is a point in its support such that $\theta (T, p) =1$, $\supp (\partial
T)\cap B_r (p)=\emptyset$ and $\|T\| (B_r (p))\leq (\omega_m +\eps)\,r^m$,
then $\supp(T)\cap B_{r/2} (p)$ is the graph of a $C^{1,\beta}$ function $f$.
\end{theorem}

Once established this $\eps$-regularity result, the regularity theory proceeds further by deriving
the usual Euler--Lagrange equations for the function $f$.
Indeed, it turns out that $f$ solves a system of elliptic partial differential equations
and the Schauder theory then implies that $f$ is smooth
(in fact analytic, using the classical result by Hopf \cite{Hopf}).

In his Big regularity paper \cite{Alm}, Almgren observes that
an intermediate regularity result can be derived as
a consequence of a more complicated construction
without using the nonparametric PDE theory of minimal surfaces
(i.e.~without deriving the 
Euler-Lagrange equation for the graph of $f$).
Indeed, given a minimizing current $T$ and a point $p$ with $\theta (T, p)= Q\in
\N$,
under the hypothesis that the excess is sufficiently small,
Almgren succeeds in constructing a $C^{3,\alpha}$ regular surface
(called {\em center manifold}) which, roughly speaking, approximates
the ``average of the sheets of the current'' (we refer to \cite{Alm} for further details).
In the introduction of \cite{Alm} it is observed that,
in the case $Q=1$, the center manifold coincides with
the current itself, thus implying directly the $C^{3,\alpha}$ regularity.

The aim of the present note is to give a simple direct proof of this remark, 
essentially following Almgren's
strategy for the construction of the center manifold in the simplified setting $Q=1$.
At this point the following comment is in order: the excess decay leading to Theorem~\ref{t:degiorgi}
remains anyway a fundamental step in the proof of this paper (see Proposition~\ref{p:degiorgi_improved} below) and, as far as we understand,
of Almgren's approach as well.
One can take advantage of the information contained in
Theorem~\ref{t:degiorgi} at several levels but 
we have decided to keep its use to the minimum. 

\section{Preliminaries}

\subsection{Some notation}\label{s:notation}
From now on we assume, without loss of generality, that $T$ is an area-minimizing
integer rectifiable current in $\R{m+n}$ satisfying the following assumptions:
\begin{equation*}
\hspace{2.5cm} \partial T = 0\; \text{in}\; B_1(0),
\quad \theta (T, 0)=1
% ,
% \quad\|T\|(B_1)\leq 1+\eps_0
\quad\text{and}\quad \|T\| (B_1)\leq \omega_m +\eps,
%\E (T, B_1(0)) \leq \eps_0.
\hspace{2.5cm}
\text{(H)}
\end{equation*}
with the small constant $\eps$ to be specified later.

In what follows, $B^m_r (q)$, $B^n_r (u)$ and $B^{m+n}_r (p)$ denote the
open balls contained,
respectively, in the Euclidean spaces $\R{m}$, $\R{n}$ and $\R{m+n}$.
Given a $m$-dimensional plane $\pi$, $\cyl^{\pi}_r (q)$ denotes the cylinder
$B^m_r (q)\times \pi^\perp \subset \pi\times \pi^\perp = \R{m+n}$ and
$\proj^\pi:\pi\times\pi^\perp\to\pi$ the corresponding orthogonal projection.
Central points, supscripts and subscripts will be often omitted when
they are clear from the context.

We will consider different systems of cartesian coordinates in $\R{m+n}$.
A corollary of De Giorgi's excess decay theorem (a variant of which is precisely
stated in Proposition \ref{p:degiorgi_improved} below) is that, when $\eps$ is sufficiently small,
the current has a unique tangent plane at the origin (see Corollary \ref{c:decay_everywhere}). Thus, immediately after
the statement of Corollary \ref{c:decay_everywhere}, the most important
system of coordinates, denoted by $x$, will be fixed once and for all in such a way that
$\pi_0=\{x_{m+1} = \ldots = x_{m+n} = 0\}$ is the tangent plane to $T$ at $0$.
Other systems of coordinates will be denoted by $x'$, $y$ or $y'$.
We will always consider positively oriented systems $x'$, i.e.~such that
there is a unique element $A\in SO (m+n)$ with $x'(p) = A \cdot x(p)$ for every point $p$.
An important role in each system of coordinates will be played by the oriented $m$-dimensional
plane $\pi$ where the last $n$ coordinates vanish (and by its orthogonal complement $\pi^\perp$).
Obviously, given $\pi$ there are several systems of coordinates $y$ for which $\pi = \{y_{m+1}= \ldots 
= y_{m+n}=0\}$. However, when we want to stress the relation between $y$ and $\pi$ we will
use the notation $y_\pi$.

\subsection{Lipschitz approximation of minimal currents}
The following approximation theorem can be found in several accounts of the regularity theory
for area-minimizing currents.
It can also be seen as a special case of a much more general result due to Almgren
(see the third chapter of \cite{Alm}) and reproved in a simpler way in
\cite{DLSp2}. As it is customary the (rescaled) 
cylindrical excess is given by the formula
\begin{equation}\label{e:cyl_ex}
\E (T, \mathcal{C}^\pi_r) := \frac{\|T\| (\mathcal{C}_r^\pi) - \omega_m
r^m}{\omega_m\,r^m} 
= \frac{1}{2\,\omega_m\,r^m} \int_{\mathcal{C}_r^\pi} |\vec{T}- \vec{\pi}|^2\,
d\|T\|, 
\end{equation}
(where $\vec{\pi}$ is the unit simple vector orienting $\pi$
and the last equality in \eqref{e:cyl_ex} holds when we assume
% $\partial T=0$  in $\mathcal{C}_r$ and 
$\proj_\sharp (T\res \mathcal{C}_r^\pi) = \a{B_r (p}$).

\begin{propos}\label{p:approx}
There are constants $C>0$ and $0<\eta,\eps_1<1$ with the following property.
Let $r>0$ and $T$ be an area-minimizing integer rectifiable $m$-current in $\cyl^\pi_{r}$ such that
\[
\de \, T=0,\quad 
\proj^\pi_\#(T)=\a{B^m_r}\quad\text{and}\quad
E := \E (T, \cyl^\pi_{r})\leq \eps_1.
\]
Then, for $s=r(1-C E^\eta)$, there exists a Lipschitz function $f: B_{s} \to \R{n}$
and a closed set $K\subset B_{s}$ such that:
\begin{subequations}
\begin{gather}
\Lip (f) \leq C E^\eta;\label{e:approx1}\\
|B_{s}\setminus K|\leq C\,r^m\, E^{1+\eta} \quad\text{and}\quad
\gr (f|_K) = T \res (K\times \R{n});\label{e:approx2}\\
\left|\|T\|(\cyl_{s}) -\omega_m\,s^m - \int_{B_{s}} \frac{|Df|^2}{2}\right|
\leq C \,r^m\,E^{1+\eta}.\label{e:approx3}
\end{gather}
\end{subequations}
\end{propos}

% 
% 
% 
% \begin{propos}\label{p:approx}
% There are constants $C,\eta,\eps_1>0$ such that the following holds.
% Let $T$ be an area-minimizing current
% 
% $p\in \supp (T)\cap B_{1/2}$, $r<1/2$ and $\pi$ so that
% $p=(q,u)\in \pi\times \pi^\perp$.
% Assume that
% \[
% \pi_\# (T\res \cyl^\pi_r (q))=\a{B_r(q)}
% \quad\text{and}\quad
% E := \E (T, B_r (p), \pi)\leq \eps_1,
% \]
% where $\pi_\#$ is the push-forward through the projection onto $\pi$.
% Then, there is a Lipschitz function $f: B_{r/2} (q)\subset \pi\to \pi^\perp$ and
% a closed set $K\subset B_{r/2} (q)$ such that:
% \begin{subequations}\label{e:approx}
% \begin{gather}
% \Lip (f) \leq C E^\eta\quad \text{and}\quad \gr_{\pi} (f)\subset B_r (p);\label{e:approx1}\\
% |B_{r/2} (q)\setminus K|\leq C E^{1+\eta} r^m
% \quad\text{and}\quad
% \gr_{\pi} (f|_K) = T \res \big(K\times B^n_{r/2}(u)\big);\label{e:approx2}\\
% \left|\|T\|\big(B_{r/2} (q)\times B^n_{r/2}(u)\big) -\omega_m  \left(\frac{r}{2}\right)^m-\frac{1}{2} \int_{B_{r/2} (q)} |Df|^2\right|
% \leq C \,E^{1+\eta} \,r^m.\label{e:approx3}
% \end{gather}
% \end{subequations}
% \end{propos}

This proposition is a key step in the derivation of Theorem~\ref{t:degiorgi}.
In the appendix we include a short proof in the spirit of \cite{DLSp2}.
Clearly, Theorem~\ref{t:degiorgi} can be thought as a much finer version of this approximation.
However, an aspect which is crucial for further developments is that
several important estimates can be derived directly from Proposition~\ref{p:approx}.

\subsection{De Giorgi's excess decay}
The fundamental step in De Giorgi's proof of Theorem~\ref{t:degiorgi} is the 
decay of the quantity usually called ``spherical excess''
(where the minimum is taken over all oriented $m$--planes $\pi$):
\begin{equation*}%\label{e:excess_coord}
\E (T, B_r (p)) := \min_\pi \E (T, B_r (p), \pi),
\quad\text{with}\quad
\E (T, B_r (p), \pi) := \frac{1}{2} \mint_{B_r (p)} |\vec{T}-\vec{\pi}|^2 d\|T\|.
\end{equation*}

\begin{propos}\label{p:degiorgi_improved}
There is a dimensional constant $C$ with the following
property. For every $\delta,\eps_0>0$, there is $\eps>0$ such that, if (H) holds, then
$\E (T, B_1)\leq \eps^2_0$ and
\begin{equation}\label{e:decay_everywhere}
\E (T, B_r (p)) \leq C\, \eps^2_0\, r^{2-2\delta}
\quad \mbox{for every $r\leq 1/2$ and every $p\in B_{1/2}\cap \supp (T)$.}
\end{equation}
\end{propos}

From now on we will consider the constant $\delta$ fixed. Its
choice will be specified much later.

\begin{definition}\label{d:admissible}
For later reference, we say that a plane $\pi$ is \textit{admissible}
in $p$ at scale $\rho$ (or simply that $(p,\rho,\pi)$ is \textit{admissible})
if
% an
% area-minimizing current $T$ as in (H) if $p\in \supp (T)\cap B_{1/2}$, $16\,\rho<1$ and
\begin{equation}\label{e:good}
\E (T, B_{\rho} (p), \pi) \leq C_{m,n}\eps^2_0\, \rho^{2-2\delta},
\end{equation}
for some fixed (possibly large) dimensional constant $C_{m,n}$.
\end{definition}

Proposition~\ref{p:degiorgi_improved} guarantees that, for every
$p$ and $r$ as in the statement, there exists always an admissible plane $\pi_{p,r}$.
The following is a straightforward consequence of Proposition~\ref{p:degiorgi_improved}
which will be extensively used.

\begin{corol}\label{c:decay_everywhere}
There are dimensional constants $C$, $C'$ and $C''$ with the following
property. For every $\delta, \eps_0>0$, there is $\eps>0$ such that,
under the assumption (H):
\begin{itemize}
\item[(a)] if $(p,\rho,\pi)$ and $(p',\rho',\pi')$ are admissible
(according to Definition \ref{d:admissible}), then
\begin{equation*}
|\vec{\pi} - \vec{\pi}'| \leq C\,\eps_0 \big(\max \{\rho, \rho', |q-q'|\}\big)^{1-\delta};
\end{equation*}
\item[(b)] there exists a unique tangent plane $\pi_p$ to $T$ at every
$p \in \supp (T)\cap B_{1/2}$; moreover, if $(p,\rho,\pi)$ is admissible then 
$|\pi-\pi_p|\leq C'\,\eps_0\,\rho^{1-\delta}$ and, vice versa,
if $|\pi-\pi_p|\leq C''\,\eps_0\,\rho^{1-\delta}$, then
$(p, \rho, \pi)$ is admissible;
\item[(c)] for every $q\in B_{1/4}^m$, there exists a unique $u\in \R{n}$ such that
$(q,u)\in \supp(T)\cap B_{1/2}$.
\end{itemize}
\end{corol}

\begin{remark}
An important point in the previous corollary is that the constant $C''$ can
be chosen arbitrarily large, provided the constant $C_{m,n}$ in Definition \ref{d:admissible}
is chosen accordingly. This fact is an easy consequence of the proof given in the appendix.
\end{remark}

Theorem~\ref{t:degiorgi} is clearly contained in the previous corollary
(with the additional feature that the H\"older exponent $\beta$ is equal to $1-\delta$,
i.e.~is arbitrarily close to $1$).
%the theorem to simplify much of the
%presentation below, but we keep its use to a minimum 
%so to get a deeper understanding of Almgren's construction.
In order to make the paper self-contained, we will include also a proof of
Proposition~\ref{p:degiorgi_improved} and Corollary~\ref{c:decay_everywhere} in the appendix.

\subsection{Two technical lemmas}
We conclude this section with the following two lemmas which will be needed in the sequel.

Consider two functions $f:D\subset\pi_0\to \pi_0^\perp$ and $f': D'\subset\pi\to \pi^\perp$,
with associated systems of coordinates $x$ and $x'$, respectively, and
$x'(p)=A\cdot x(p)$ for every $p\in\R{m+n}$.
If for every $q'\in D'$ there exists a unique $q\in D$ such that $(q', f'(q')) = A \cdot (q, f(q))$ and vice versa, then
it follows that $\gr_{\pi_0}(f)=\gr_\pi(f')$, where
\[
\gr_{\pi_0} (f) := \big\{(q, f(q))\in D\times \pi_0^\perp\big\}
\quad\text{and}\quad
\gr_\pi (f') := \big\{(q', f'(q'))\in D'\times \pi^\perp\big\}.
\]
The following lemma compares norms of functions (and of differences of functions)
having the same graphs in two nearby system of coordinates.

\begin{lemma}\label{l:rotation}
There are constants $c_0,C>0$ with the following properties.
Assume that
\begin{itemize}
\item[(i)] $\|A-{\rm Id}\|\leq c_0$, $r\leq 1$;
\item[(ii)] $(q,u)\in\pi_0\times\pi_0^\perp$ is given and $f,g: B^m_{2r} (q)\to \R{n}$ are Lipschitz functions such that
\begin{equation*}
\Lip (f), \Lip (g) \leq c_0\quad\text{and}\quad |f(q)-u|+|g(q)-u|\leq c_0\, r.
\end{equation*}
\end{itemize}
Then, in the system of coordinates $x'= A\cdot x$, for $(q',u') = A \cdot(q,u)$, the following holds:
\begin{itemize}
\item[(a)] $\gr_{\pi_0} (f)$ and $\gr_{\pi_0} (g)$ are the graphs of two Lipschitz functions $f'$ and $g'$, whose domains of definition contain both $B_{r} (q')$;
\item[(b)] $\|f'-g'\|_{L^1 (B_{r} (q'))}\leq C\,\|f-g\|_{L^1 (B_{2r} (q))}$;
\item[(c)] if $f\in C^4 (B_{2r} (q))$, then $f'\in C^4 (B_{r} (q'))$, with the
estimates 
\begin{eqnarray}
\|f'- u'\|_{C^3}&\leq& \Phi \left(\|A-\Id\|, \|f-u\|_{C^3}\label{e:est_C3}\right)\, ,\\
\|D^4 f'\|_{C^0} &\leq& \Psi \left(\|A-\Id\|, \|f-u\|_{C^3}\right) 
\left(1+ \|D^4 f\|_{C^0}\right)\, ,
\label{e:est_C4}
\end{eqnarray}
where $\Phi$ and $\Psi$ are smooth functions.
\end{itemize}
\end{lemma}
\begin{proof} Let $P: \R{m\times n}\to \R{m}$ and
$Q: \R{m\times n}\to \R{n}$ be the usual orthogonal projections.
Set $\pi=A(\pi_0)$ and
consider the maps $F, G: B_{2r} (q)\to \pi^\perp$ and $I, J: B_{2r} (q)\to \pi$
given by
\[
F (x) = Q (A((x,f(x)))\quad\text{and}\quad G(x) = Q (A((x, g(x))),
\]
\[
I(x)= P (A((x, f(x)))\quad\text{and}\quad J (x) = P (A((x, g(x))).
\]
Obviously, if $c_0$ is sufficiently small, $I$ and $J$ are injective Lipschitz maps.
Hence, $\gr_{\pi_0} (f)$ and $\gr_{\pi_0} (g)$ coincide, in the new coordinates, with the graphs of the functions
$f'$ and $g'$ defined respectively in $D:= I (B_{2r} (q))$ and $\tilde{D}:= J (B_{2r} (q))$
by $f' = F \circ I^{-1}$ and $g'= G \circ J^{-1}$.
If $c_0$ is chosen sufficiently small, then we can find a constant $C$
such that 
\begin{equation}\label{e:Lip_bound}
\Lip (I), \; \Lip (J),\; \Lip (I^{-1}),\; \Lip (J^{-1}) \leq 1+C\,c_0,
\end{equation}
and
\begin{equation}\label{e:Linfty_bound}
|I (q)-q'|, |J (q)-q'|\leq C\,c_0\, r.
\end{equation}

Clearly, \eqref{e:Lip_bound} and \eqref{e:Linfty_bound} easily imply (a).
Conclusion (c) is a simple consequence of the inverse function theorem.
Finally we claim that, for small $c_0$,
\begin{equation}\label{e:claim}
|f'(x')-g'(x')|\leq 2 \,|f (I^{-1} (x')) - g (I^{-1} (x'))|
\quad\forall \;x'\in B_r(q'),
\end{equation}
from which,
using the change of variables formula for biLipschitz homeomorphisms
and \eqref{e:Lip_bound}, (b) follows.

In order to prove \eqref{e:claim}, consider 
any $x'\in B_r (q')$, set $x:= I^{-1} (x')$ and
\[
p_1 := (x, f(x))\in \pi_0\times \pi_0^\perp,\quad
p_2 := (x, g(x))\in \pi_0\times \pi_0^\perp\quad
\text{and}\quad
p_3 := (x', g'(x'))\in \pi\times \pi^\perp.
\]
Obviously $|f'(x')- g'(x')|= |p_1-p_3|$ and $|f(x)- g(x)|=|p_1-p_2|$.
Note that, $g(x)= f (x)$ if and only if $g'(x')= f' (x')$, and in this case \eqref{e:claim} follows trivially.
If this is not the case, the triangle with vertices $p_1$, $p_2$ and $p_3$ is non-degenerate.
Let $\theta_i$ be the angle at $p_i$.
Note that, $\Lip (g)\leq c_0$ implies $|90^\circ-\theta_2|\leq C c_0$ and $\|A-{\rm Id}\|\leq c_0$ implies
$|\theta_1|\leq C c_0$, for some dimensional constant $C$.
Since $\theta_3 = 180^\circ - \theta_1 - \theta_2$, we conclude
as well $|90^\circ - \theta_3|\leq C c_0$.
Therefore, if $c_0$ is small enough, we have $1 \leq 2\sin \theta_3$, so that,
by the Sinus Theorem,
\begin{equation*}
|f'(x')-g'(x')|= |p_1-p_3| = \frac{\sin \theta_2}{\sin \theta_3}\, |p_1-p_2|
\leq 2 \, |p_1-p_2| = 2 \,|f(x)-g(x)|,
\end{equation*}
thus concluding the claim.
\end{proof}

The following is an elementary lemma on polynomials.

\begin{lemma}\label{l:poly}
For every $n,m\in \N$, there exists a constant $C(m,n)$ such that,
for every polynomial $R$ of degree at most $n$ in $\R{m}$ and every positive $r>0$,
\begin{equation}\label{e:poly}
|D^k R (q)| \leq \frac{C}{r^{m+k}} \int_{B_r (q)} |R|\, \quad \text{for all }\; k\leq n\; \text{ and all }\;q\in\R{m}.
\end{equation}
\end{lemma}
\begin{proof} We rescale and translate the variables by setting
$S (x) = R (rx+q)$. The lemma is then reduced to show that
\begin{equation}\label{e:norms}
\sum_{k=0}^n |D^k S (0)|\leq C \int_{B_1 (0)} |S|,
\end{equation}
for every polynomial $S$ of degree at most $n$ in $\R{m}$, with $C=C(n,m)$.
Consider now the vector space $V^{n,m}$ of polynomials of degree
at most $n$ in $m$ variables. $V^{n,m}$ is obviously finite dimensional.
Moreover, on this space, the two quantities
\[
\|S\|_1:= \sum_{k=0}^n |D^k S (0)| \quad \text{and}\quad
\|S\|_2:= \int_{B_1 (0)} |S|
\]
are two norms.
The inequality \eqref{e:norms} is then a corollary
of the equivalence of norms on finite-dimensional vector spaces.
\end{proof}

\section{The approximation scheme and the main theorem}

The $C^{3,\alpha}$ regularity of the current $T$ will be deduced from
the limit of a suitable approximation scheme. In this section we describe
the scheme and state the main theorem of the paper.

We start by fixing a nonnegative kernel $\varphi\in C^\infty_c (B^m_1)$ 
which is radial and satisfies $\int\varphi =1$.
As usual, for $\tau>0$, we set $\varphi_\tau (w) := \tau^{-m} \varphi (w/\tau)$.
Consider the area-minimizing current
$S=T\res(B_{1/2}^{m+n}\cap\cyl^{\pi_0}_{1/4})$
(recall that $\pi_0 = \{x_{m+1} = \ldots = x_{m+n} = 0\}$ is the tangent
plane to $T$ at $0$).
% From Corollary~\ref{c:decay_everywhere} (c), it follows that
% \[
% \proj_\#(S)=\a{B_{1/4}}
% \quad\text{and}\quad
% \supp(\de S)\subset\de \cyl_{1/4}.
% \]
% Moreover, always 
From Corollary~\ref{c:decay_everywhere} (b) and (c),
it is simple to deduce the following: if $p=(q,u)\in\pi\times\pi^\perp$,
$\rho\leq 2^{-6}$ and $\pi$ form an admissible triple $(p,8\,\rho,\pi)$ 
with $p\in \supp(T)\cap B_{1/16}$, then 
\[
\proj^\pi_\#(S\res \cyl^\pi_{8\rho}(q))=\a{B_{8\rho}(q)}
\quad\text{and}\quad
\de S=0\quad\text{in}\quad \cyl^\pi_{8\rho}(q).
\]
From now on, we will assume that $C_{m,n}\,\eps^2_0\,2^{-3(2-2\delta)}\leq
\eps_1$,
where $\eps_1$ is the constant of Proposition~\ref{p:approx} and $C_{m,n}$ the 
constant of Proposition~\ref{p:degiorgi_improved}.
This assumption guarantees the existence of the Lipschitz approximation of
Proposition~\ref{p:approx},
which we restrict to $B^m_{6\rho} (q)$,
$f: B^m_{6\rho} (q)\subset\pi \to \pi^\perp$.
Then, consider the following functions:
\begin{itemize}
\item[($I_1$)] $\hat{f} = f * \varphi_{\rho}$;
\item[($I_2$)] $\bar{f}$ such that
\begin{equation*}
\left\{
\begin{array}{l}
\Delta \bar{f} = 0 \quad \text{on }\,B^m_{4\rho} (q),\\
\bar{f}|_{\partial B^m_{4\rho} (q)} = \hat{f};
\end{array}\right.
\end{equation*}
\item[($I_3$)] $g: B^m_\rho (q')\subset\pi_0\to \pi_0^\perp$,
with $x(p)=(q',u')\in\pi_0\times\pi_0^\perp$, such that
$\gr_{\pi_0}(g)=\gr_{\pi}(\bar f)$ in the cylinder $\cyl_\rho (q')\subset\pi_0\times \pi_0^\perp$.
\end{itemize}

\begin{remark}
In order to proceed further, we need to show the existence of $g$ as in ($I_3$).
We wish, therefore, to apply Lemma \ref{l:rotation} to the function $\bar{f}$.
First recall that $|\pi-\pi_0|\leq C\eps_0|p|^{1-\delta}\leq C\eps_0$ by
Corollary \ref{c:decay_everywhere}. Thus, assumption (i) in Lemma \ref{l:rotation}
is satisfied provided $\eps_0$ is chosen sufficiently small.
Next note that, by the interior estimates for the harmonic functions and \eqref{e:approx1}, one has
\[
\Lip(\bar f|_{B_{3\rho}})\leq C\Lip(\hat f|_{B_{4\rho}})\leq C\,E^\eta\, .
\]
Moreover, if we consider the ball $B_s (p)$ with $s= \rho E^{\eta/(2m)}$, by the monotonocity
formula, $\|T\| (B_s (p))\geq \omega_m \rho^m E^{\eta/2}$. Thus, by \eqref{e:approx2},
the graph of $f$ contains a point in $B_s (p)$. Using the Lipschitz bound \eqref{e:approx1},
we then achieve $\|f-u\|_{C^0 (B_{6\rho} (q))}\leq C \rho E^{\eta/2}$, which in turn implies
$\|\bar{f} - u\|_{C^0 (B_{4\rho} (q))}$. Recalling that $E\leq C \eps^2_0 \rho^{2-2\delta}$,
we conclude that condition (ii) in Lemma \ref{l:rotation} is satisfied when
$\eps_0$ is sufficiently small.
Therefore Lemma \ref{l:rotation}(a) guarantees that the function $g$ exists.
\end{remark}

\begin{remark}
It is obvious that in order to define the function $g$ we could have used,
in place of the $f$ given by Proposition \ref{p:approx},
the function whose graph gives the current $T$ in $B_{6\rho} (p)$.
This would have simplified many of the computations below. However, as mentioned
in the introduction, we hope that our choice helps in the understanding of
the more general construction of Almgren. 
\end{remark}

The function $g$ is the main building block of the
construction of this paper. It is called the {\em $(p,\rho,\pi)$-interpolation of $T$} or,
if $\E (T, B_{8\rho} (p)) = \E (T, B_{8\rho} (p),\pi)$,
simply the {\em $(p, \rho)$-interpolation of $T$}.

The main estimates of the paper are contained in the following proposition.

\begin{propos}\label{p:main}
There are constants $\alpha,C>0$ such that, if $g, g'$ are respectively $(p,\rho,\pi)$- and
$(p',\rho,\pi')$-interpolations, then
\begin{subequations}\label{e:ABC}
\begin{gather}
\rho^{1-\alpha} \|D^4 g\|_{C^0}+ \|g\|_{C^3}\leq C,\label{e:A}\\
\sum_{\ell=0}^4
\rho^{\ell-3-\alpha}\|D^\ell g (x) - D^\ell g' (x)\|_{C^0}\leq C \quad\text{in }\;B_\rho(p)\cap B_{\rho}(p'),\label{e:B}\\
|D^3 g (q) - D^3 g' (q')|\leq C |q-q'|^\alpha,
\quad\text{with }\;p=(q,u),\,p'=(q',u').\label{e:C}
\end{gather}
\end{subequations}
\end{propos}

\subsection{Approximation scheme}
Let $5<n_0<k_0$ be natural numbers and consider the cube 
$Q=[-2^{-n_0}, 2^{-n_0}]^m$.
For $k\geq k_0$, we consider the usual subdivision
of $\R{m}$ into dyadic cubes of size $2\cdot 2^{-k}$, centered 
at points $c_i = 2^{-k} i \in 2^{-k} \Z^m$.
The corresponding closed cubes of the subdivision are then denoted by $Q_i$ and 
we consider below only those $Q_i$'s which have nonempty intersection 
with $Q$.

According to Corollary~\ref{c:decay_everywhere} and to the previous observations,
for every $c_i$ there exists a unique $u_i$ such that
$p_i = (c_i, u_i)\in \supp (T)\cap B_{1/16}$. Moreover,
for every constant $C$, if $k_0$ is large enough,
we can consider the $(p_i, C\, 2^{-k})$-interpolation $g_i$
for all $k\geq k_0$.

Let $\psi\in C^\infty_c ([-\textstyle{\frac{5}{4}},
\textstyle{\frac{5}{4}}]^m)$ be a nonnegative function such that,
if we define $\psi_i (q):= \psi (2^{k} (q-c_i))$, then
\[
\sum_{i\in \Z^{m}} \psi_i \equiv1\;\text{in}\;Q\, .
\]
Denote by $\mathcal{A}_i$ the set of indices $j$
such that $Q_j$ and $Q_i$ are adjacent. 
Note that the choice of $\psi$ guarantees $\psi_i\,\psi_j=0$ if $j\not\in \mathcal{A}_i$. 
Moreover the cardinality of $\mathcal{A}_i$ is (bounded by) a
dimensional constant independent of $k$ and, if $q\in Q_i$, then
in a neighborhood of $q$ we have
\begin{equation}\label{e:sum_der=0}
\sum_{j\in \mathcal{A}_i} \psi_j=1
\quad\text{and}\quad
\sum_{j\in \mathcal{A}_i} D^\ell \psi_j (q) = 0 \quad
\text{for all }\;\ell>0.
\end{equation}
We are now ready to state and prove the central theorem of this note.

\begin{theorem}\label{t:main}
There are dimensional constants $n_0<k_0$ with the following
properties.
Given an area-minimizing current $T$ as in (H) and $k\geq k_0$,
consider the functions $h_k : {Q} \to \R{n}$ given by
$h_k := \sum_i \psi_i\, g_i$. Then,
\begin{equation}\label{e:main_estimate}
\|h_k\|_{C^{3,\alpha}}\leq C,
\end{equation} 
for some dimensional constants $\alpha>0$ and $C$ (which, in particular,
do not depend on $k$).
Moreover, the graphs of $h_k$ converge, in the sense of currents, to 
$T\res (Q\times \R{n})\cap B_{1/2}$, thus implying that $T$ is a $C^{3,\alpha}$ graph in
a neighborhood of the origin.
\end{theorem}

\begin{proof}[Proof of Theorem~\ref{t:main}]
Given $k$, consider a cube $Q_i$ of the corresponding dyadic decomposition and a point $q\in Q_i$.
We already observed that, in a neighborhood of $q$, 
$h_k = \sum_{j\in \mathcal{A}_i} \psi_j g_j$.
Moreover, from the definition, we have that
\begin{equation}\label{e:psi}
\|D^\ell \psi_j\|_{C^0}= 2^{k\ell}\, \|D^\ell\psi\|_{C^0}= C_\ell\, 2^{k\ell} \quad\text{for every }\;\ell\in\N.
\end{equation}
The $C^0$ estimate of $h_k$ follows trivially from \eqref{e:A}, since
\[
|h_k (q)|\leq \sum_{j\in\mathcal{A}_i} \|\psi_j\|_{C^0}\|g_j\|_{C^0}
\leq C.
\]
As for the $C^1$ estimate, we write the first derivative of $h_k$ as follows,
\[
D h_k (q) = \sum_{j\in \mathcal{A}_i} \big( D \psi_j (q) g_j (q) +
\psi_j (q) D g_j (q)\big) \stackrel{\eqref{e:sum_der=0}}{=} \sum_{j\in \mathcal{A}}
\big( D \psi_j (q) (g_j (q)- g_i (q)) + \psi_j (q) D g_j (q)\big),
\]
from which, using \eqref{e:A}, \eqref{e:B} and \eqref{e:psi}, we deduce
\[
|D h_k (q)|\leq \sum_{j\in \mathcal{A}_i} \big(\|D\psi_j\|_{C^0} \|g_i -g_j\|_{C^0}
+ \|\psi_j\|_{C^0} \|Dg_j\|_{C^0}\big)\leq C.
\]
With analogous computations, we obtain
\begin{align*}
|D^2 h_k (q)|\leq{}& \sum_{j\in\mathcal{A}_i} \big(\|D^2 \psi_j\|_{C^0}
\|g_i-g_j\|_{C^0} + \|D \psi_j\|_{C^0} \|Dg_j - Dg_i\|_{C^0}+\|\psi_j\|_{C^0} \|D^2 g_j\|_{C^0}\big)\leq C,\\
% \end{equation*}
% \begin{align*}
|D^3 h_k (q)|\leq{}& \sum_{j\in\mathcal{A}_i} \big(\|D^3 \psi_j\|_{C^0}
\|g_j-g_i\|_{C^0} + \|D^2 \psi_j\|_{C^0} \|Dg_j - Dg_i\|_{C^0}+\\
&+ \|D \psi_j\| \|D^2 g_j - D^2 g_i\|_{C^0} + \|\psi_j\|_{C^0} \|D^3 g_j\|_{C^0}\big)\leq C,\\
% \end{align*}
% \begin{align*}
|D^4 h_k (q)|\leq{}& \sum_{j\in\mathcal{A}_i} \big(\|D^4 \psi_j\|_{C^0}
\|g_j-g_i\|_{C^0} + \|D^3 \psi_j\|_{C^0} \|Dg_j - Dg_i\|_{C^0}+\\
&+ \|D^2 \psi_j\| \|D^2 g_j - D^2 g_i\|_{C^0} + \|D\psi_j\|_{C^0} \|D^3 g_j-D^3g_i\|_{C^0}+\|\psi_j\|_{C^0} \|D^4 g_j\|_{C^0}\big)\\
\leq {}& C\,2^{k(1-\alpha)},
\end{align*}
where $C$ is a constant independent of $k$.

Now, let $q,q'\in B_{1/2}$ and consider the cubes $Q_i$ and $Q_j$ such that
$q\in Q_i$ and $q'\in Q_j$.
If the two cubes are adjacent, then we have $|q-q'|\leq C 2^{-k}$ and, therefore,
\[
|D^3 h_k (q) - D^3 h_k (q')|\leq \|D^4 h_k\|_{C^0} |q-q'|\leq C \,2^{k(1-\alpha)}\, |q-q'| \leq C\, |q-q'|^\alpha.
\]
If $Q_i$ and $Q_j$ are not adjacent, then $2\,|q-q'|\geq \max\{|c_i-c_j|, 2^{-k}\}$.
Since $\supp(\psi)\subset [-\textstyle{\frac{5}{4}},
\textstyle{\frac{5}{4}}]^m$, $D^3 h_k (c_i)= D^3 g_i (c_i)$ for every $i$ and
from \eqref{e:C} it follows that
\begin{align*}
|D^3 h_k (q)-D^3 h_k (q')| \leq{}&
|D^3 h_k (q)-D^3 h_k (c_i)| + |D^3 g_i (c_i)-D^3 g_j (c_j)|+\\
&+ |D^3 h_k (c_j) - D^3 h_k (q')|\\
\leq{}& C \,2^{-k} \|D^4 h_k\|_{C^0}
+ C |c_j-c_i|^\alpha
\leq C 2^{-k\alpha} + C |c_i-c_j|^\alpha\leq C |q-q'|^\alpha.
\end{align*}
This concludes the proof of \eqref{e:main_estimate}.
We finally come to the convergence of the graphs of $h_k$ in the sense of currents.
Obviously, by compactness we can assume that a subsequence of $h_k$
(not relabelled) converges in
the $C^3 (Q)$ norm to some limiting $C^{3,\alpha}$ function $h$. 
On the other hand, by Corollary \ref{c:decay_everywhere}
and Proposition \ref{p:approx}, it follows easily that the support of
$T\res (Q\times\R{n})\cap B_{1/2}$ is contained in the graph of $h$.
But then, by the Constancy Theorem, $T\res (Q\times\R{n})\cap B_{1/2}$
must coincide with an integer multiple of the graph of $h$. Our assumptions
imply easily that the multiplicity is necessarily $1$.
\end{proof}

\section{$L^1$-estimate}
The rest of the paper is devoted to the proof of Proposition~\ref{p:main}.
A fundamental point is
an estimate for the $L^1$ distance between the harmonic function $\bar f$
introduced in step ($I_2$) of the approximation scheme
and the function $f$ itself.
A preliminary step is the following estimate on the Laplacian of $\hat f$, which
is a simple consequence of the first variation formula for area-minimizing currents.

\begin{lemma}\label{l:laplace}
There exists $\delta, \gamma, C, \lambda>0$ such that, if
$(p,8\rho, \pi)$ is admissible and $\hat{f}$ is as in ($I_1$), then
\begin{equation}\label{e:laplace}
\|\Delta \hat{f}\|_{C^0 (B^m_{5\rho})}\;\leq\; C \rho^{1+\lambda},
\end{equation}
\begin{equation}\label{e:laplace0}
\int_{B_{5\rho}}\bigg\vert \int_{B_\rho(w)} Df(z)\cdot D\gamma 
(w-z) \,dz\bigg\vert\,dw\leq 
C\, E^{1+\eta}\,\rho^m\,\|D\gamma\|_{L^1},
\quad\forall\gamma\in C_c^1(B_{\rho},\R{n}),
\end{equation}
where $\eta$ is the constant in Proposition~\ref{p:approx}
and $E =\E (T, B_{8\rho} (p), \pi)$.
\end{lemma}

\begin{proof}
Let $\mu$ be the measure defined by $\mu(A):=\|T\|(A\times\pi^{\perp})$.
We start showing that the approximation $f$ given by
Proposition~\ref{p:approx} satisfies
\begin{equation}\label{e:variation}
\left\vert\int Df\cdot D\kappa\right\vert
\leq C\int |D\kappa|\,|Df|^3\,dx+C\int |D\kappa|\,{\bf 1}_{B_{6\rho}\setminus K}\, (dx+ d\mu(x)),
\end{equation}
for every $\kappa\in C_c^1(B_{6\rho},\R{n})$.
Consider the vector field $\chi(x,y)=(0,\kappa (x))$.
From the minimality of the current $T$, we infer that the first variation of the mass in direction $\chi$
vanishes, $\delta T(\chi)=0$.
We set $T_f=\gr(f)$. Since $\delta T (\chi)=0$, we get
\begin{equation}\label{e:variation2}
\left\vert\int Df\cdot D\ph\right\vert\leq
\left\vert\int Df\cdot D\ph-\delta T_f(\chi)\right\vert+
\left\vert \delta T(\chi)-\delta T_f(\chi)\right\vert.
\end{equation}
The first variation $\delta T_f (\chi)$ is given by the formula
\begin{align*}
\int_{\cyl_{6\rho}}\dv_{\vec T_f}\chi \, d\|T_f\|&=\frac{d}{ds}\Big\vert_{s=0}
\int_{B_{6\rho}}\sqrt{1+|D f+sD \kappa|^2+\textstyle{\sum_{|\alpha|\geq2}}M_\alpha(D f+sD \kappa)^2}\,dx\\
&=\int_{B_{6\rho}}\frac{D f\cdot D\kappa+\textstyle{\sum_{|\alpha|\geq2}}M_\alpha(D f)\left.\frac{d}{ds}\right|_{s=0} M_{\alpha}(D f+sD \kappa)}{\sqrt{1+|D f|^2+\sum_{|\alpha|\geq2}M_\alpha(D f)^2}}.
\end{align*}
It follows then that
\begin{align*}
\left\vert\int_{\cyl_{6\rho}}\hspace{-0.2cm}\dv_{\vec T_f}\chi \, d\|T_f\|-
\int_{B_{6\rho}}\hspace{-0.2cm}Df\cdot D\kappa\right\vert\leq{}&
\int_{B_{6\rho}}\hspace{-0.2cm}|D f| |D\kappa|\left(\sqrt{1+|D f|^2+\textstyle{\sum_{|\alpha|\geq 2}} M_\alpha(D f)^2}-1\right)+\\
& +\left\vert\int_{B_{6\rho}}\textstyle{\sum_{|\alpha|\geq2}}M_\alpha(D f)\left.\frac{d}{ds}\right|_{s=0} M_{\alpha}(D f+sD \kappa)\right|\\
\leq {}& C\int_{B_{6\rho}}|D\kappa|\,|Df|^3.
\end{align*}
We next estimate the second term in the right hand side of 
\eqref{e:variation2}:
\begin{align*}
\big\vert \delta T(\chi)-\delta T_f(\chi)\big\vert&\leq
\int_{B_{6\rho}\setminus K\times\R{n}}\dv_{\vec T}\chi \, d\|T\|
+\int_{B_{6\rho}\setminus K\times\R{n}}\dv_{\vec T_f}\chi \, d\|T_f\|\\
&\leq \int_{}{\bf 1}_{B_{6\rho}\setminus K}(x)|D\kappa|(x)\, d\mu(x)
+ C\int_{}{\bf 1}_{B_{6\rho}\setminus K}(x)|D\kappa|(x) \, dx, 
\end{align*}
where we have used
the Lipschitz bound on $f$ to estimate the second integral in the right hand
side of the first line.
This concludes the proof of \eqref{e:variation}.

We now come to the proof of \eqref{e:laplace}.
From \eqref{e:variation} and Proposition~\ref{p:approx}, it follows straightforwardly that
\begin{equation}\label{e:step1}
\left\vert\int_{B_{6\rho}}D f\cdot D\kappa\right\vert\leq C\, E^{1+\eta}\,\rho^m\,\|D\kappa\|_{L^\infty},
\quad\text{for every }\;\kappa\in C_c^1(B_{6\rho},\R{n}).
\end{equation}
Then, putting together the previous estimates, we conclude that
\begin{align*}
\|\Delta \hat f\|_{L^{\infty}(B_{5\rho})}&=\sup_{\gamma\in C^1_c(B_{5\rho}),\|\gamma\|_{L^1}\leq1}\int D\hat f\cdot D\gamma
=\sup_{\gamma\in C^1_c(B_{5\rho}),\|\gamma\|_{L^1}\leq1}
\int D f\cdot D(\gamma*\ph_{\rho})\\
&\stackrel{\mathclap{\eqref{e:step1}}}{\leq}\sup_{\gamma\in C^1_c(B_{5\rho}),\|\gamma\|_{L^1}\leq1} C\,E^{1+\eta}\rho^m\,\|D(\gamma*\ph_\rho)\|_{L^\infty}
\leq C\, E^{1+\eta}\,\rho^m\,\|D\ph_{\rho}\|_{L^\infty}\\
&\leq C\, E^{1+\eta}\, \rho^{-1}\leq C\,\rho^{(2-2\delta)(1+\eta)-1}.
\end{align*}
Therefore, \eqref{e:laplace} follows choosing $\delta$ sufficiently small
with respect to $\eta$.

For the proof of \eqref{e:laplace0}, it is enough to notice that, from
\eqref{e:variation} and Proposition~\ref{p:approx}, we get
\begin{align*}
&\int_{B_{5\rho}}\left\vert \int_{B_{\rho} (w)} Df(z)\cdot D\gamma (w-z)
\,dz\right\vert\,dw\\
\leq& C\int_{B_{5\rho}} 
|D\gamma|*|Df|^3+C\int_{B_{5\rho}} |D\gamma|*{\bf 1}_{\R{m}\setminus
K}+C\int_{B_{5\rho}} |D\gamma|*\big( \mu\res(\R{m}\setminus K)\big)\\
\leq& \|D\gamma\|_{L^1} \left( C E^\eta\,\int_{B_{6\rho}}|Df|^2+ |B_{6\rho}\setminus K|
+ \mu(B_{6\rho}\setminus K)\right)
\leq C\, E^{1+\eta}\,\rho^m\,\|D\gamma\|_{L^1}.
\end{align*}
\end{proof}

Now we come to the $L^1$-estimate for the harmonic approximation $\bar f$.

\begin{propos}\label{p:L1}
Let $(p,8\rho,\pi)$ be admissible and $\bar{f}$ be as in ($I_2$).
Then, there exists $\alpha>0$ such that
\begin{equation}\label{e:L1prov}
\|\bar f-f\|_{L^1(B_{4\rho})}\leq  C\, \rho^{m+3+\alpha}.
\end{equation}
\end{propos}

\begin{proof}
First we estimate the $L^1$ distance between $\bar f$ and $\hat f$.
Using the Poincar\'e inequality and a simple integration by parts, we infer that
\begin{equation*}
\|\bar f-\hat{f}\|_{L^1 (B_{4\rho})}^2\leq C \, \rho^{m+2}\, 
\|\nabla (\bar f-\hat{f})\|_{L^2(B_{4\rho})}^2
= C \,\rho^{m+2} \int_{B_{4\rho}} \Delta \hat{f} \,(\bar{f} - \hat{f}),
\end{equation*}
from which
\begin{equation*}%\label{e:L1prov3}
\|\hat{f}-\bar{f}\|_{L^1 (B_{4\rho})}\leq
C \,\rho^{2+m} \,\|\Delta \hat{f}\|_\infty
\stackrel{\eqref{e:laplace}}{\leq} C \,\rho^{m+3+\lambda}.
\end{equation*}
In order to prove \eqref{e:L1prov}, then it is enough to prove the following inequality,
\begin{equation}\label{e:Linfty}
\|\hat{f} - f\|_{L^1 (B_{4\rho})} \leq C \rho^{m+3+\alpha}.
\end{equation}
For every $z\in B_{4\rho}$, from the definition of $\hat{f}$ we have
\begin{equation}\label{e:explicit}
\hat{f} (z)-f (z) = \int \varphi_\rho (z-y) (f(y)-f (z))\, dy.
\end{equation}
To simplify the notation assume $z=0$ and rewrite \eqref{e:explicit} as
\begin{align*}
\hat{f} (0) -f (0) &= \int \varphi_\rho (y) \int_0^{|y|}
\frac{\partial f}{\partial r} \left( \tau \frac{y}{|y|}\right)\, d\tau\, dy
= \int \varphi_\rho (y) \int_0^{|y|} \nabla f \left(\tau\frac{y}{|y|}\right)
\cdot \frac{y}{|y|}\, d\tau\, dy\notag\\
&= \int \varphi_\rho (y) \int_0^1 \nabla f (\sigma y) \cdot y\, d\sigma\, dy
= \int \int_0^1 \varphi_\rho \left(\frac{w}{\sigma}\right)\, \nabla f (w) \cdot
\frac{w}{\sigma^{m+1}}\, d\sigma\, dw\notag\\
&= \int \nabla f (w) \cdot \underbrace{w \left(\int_0^1 \varphi_\rho
\left(\frac{w}{\sigma}\right)\,\sigma^{-m-1}\, d\sigma\right)}_{=: \Phi (w)}\, dw .%\label{e:explicit2}
\end{align*}
More generally, for every $z\in B_{4\rho}$, we have
$\hat{f} (z) -f (z)= \int \nabla f (w) \cdot \Phi (w-z)\, dw$
% \begin{equation*}
% \hat{f} (z) -f (z)= \int \nabla f (w) \cdot \Phi (w-z)\, dw ,
% \end{equation*}
and
\begin{equation*}
\|\hat f-f\|_{L^1(B_4\rho)}=\int_{B_{4\rho}}\left\vert \int \nabla f (w) \cdot \Phi (w-z)\, dw\right\vert\,dz.
\end{equation*}
Since $\varphi$ is radial, the function $\Phi$ is a gradient.
Indeed, it can be easily checked that, for any $\psi$, the vector field $\psi (|w|)\, w$ is curl-free.
Moreover, $\supp (\Phi)$ is compactly contained in $B_\rho$.
Hence, we can apply \eqref{e:laplace0} and get
\begin{equation}\label{e:norm1}
\|\hat f-f\|_{L^1(B_4\rho)}\leq C\, E^{1+\eta}\,\rho^m\, \|\Phi\|_{L^1}.
\end{equation}
By a simple computation,
\begin{align*}
\|\Phi\|_{L^1}=\int_{\R{m}}\int_0^1 |w|\,\varphi
\left(\frac{w}{\rho\sigma}\right)\,\rho^{-m}\sigma^{-m-1}\, d\sigma\,dw
=\rho \int_{\R{m}}\int_0^1|y|\,\ph(y)\,d\sigma\,dy\, .
\end{align*}
The last integral is a constant which depends only on $\ph$.
Thus, \eqref{e:Linfty} follows from \eqref{e:norm1}.
\end{proof}

A simple consequence of the $L^1$-estimate is a comparison between
harmonic approximations at different scales.

\begin{corol}\label{c:alpha}
Assume $(p, 16\,r, \pi)$ is an admissible triple and
let $\bar f_1$ and $\bar f_2$ be as in $(I_2)$, with $\rho=r$ and $\rho=2\,r$ respectively.
Then, if $p=(q,u)\in \pi\times \pi^\perp$,
\begin{equation}\label{e:stima_alpha}
\sum_{\ell=0}^4 r^{\ell-3-\alpha}\|D^\ell \bar f_1-D^\ell \bar{f}_2\|_{C^0 (B^m_{3r/2} (q))}
\leq C .
\end{equation}
%(This conclusion implies a choice of the constant $\delta$ which is possibly smaller
%than what made so far, but only depending on $\eta$ and $m$).
\end{corol}

\begin{proof}
It is enough to show that
\begin{equation}\label{e:L1bis}
\|\bar f_1-\bar f_2\|_{L^1(B_{2r})}\leq C\, r^{m+3+\alpha},
\end{equation}
because then the conclusion of the lemma follows easily from the classical
mean-value property of harmonic functions.
Clearly, from the admissibility of $(p,16\,r,\pi)$ and
Corollary~\ref{c:decay_everywhere},
it follows that $|\pi-\pi_p|\leq C\,r^{2-2\delta}$.
Hence, always by the same corollary
$E_2:=\E(T, B_{16 r}(p), \pi)\leq C r^{2-2\delta}$.
Then, in view of Proposition \ref{p:L1}, in order to show \eqref{e:L1bis},
it suffices to prove
\begin{equation}\label{e:L1prov2}
\|f_1-f_2\|_{L^1(B_{2r})}\leq  C\, r^{m+3+\alpha}.
\end{equation}
Note first that $f_1$ and $f_2$ coincide
on a set $K$ with $|B_{2r} \setminus K|\leq C E_2^{1+\eta} r^m$.
Moreover, since the Lipschitz constants of $f_1$ and $f_2$ are bounded by
a universal constant $C$,
we have $|f_1 (z)- f_2 (z)|\leq C r$ for every $z\in B_{2r}$.
Therefore, we conclude \eqref{e:L1prov2} from
\begin{equation*}%\label{e:L1f_1f_2}
\|f_1-f_2\|_{L^1 (B_{2r})} \leq C r |B_{2r}\setminus K|\leq C \,r\, E_2^{1+\eta}\,r^m
\leq C \, r^{m+1+(1+\eta)(2-2\delta)}.
\end{equation*}
%(provided $\delta$ is chosen sufficiently small, depending on $\eta$ and $m$).
\end{proof}

\section{Proof of Proposition~\ref{p:main}}
The proof of \eqref{e:A} in Proposition~\ref{p:main} is given by a simple iteration
of Corollary \ref{c:alpha} on dyadic balls.

\begin{lemma}\label{l:beta}
Let $g_1, g_2$ be respectively the $(p, \rho, \pi)$- and the
$(p, 2^N\rho, \pi)$-interpolation (under the assumption of admissibility \eqref{e:good}).
Then, for $p=(q',u')\in \pi_0\times\pi_0^\perp$, it holds
\begin{equation}\label{e:A'}
\|g_1\|_{C^3} + \rho^{1-\alpha} \|D^4 g_1\|_{C^0}\leq C ,
\end{equation}
\begin{equation}\label{e:stima_beta}
|D^3 g_1 (q') -D^3 g_2 (q')| \leq C (2^N \rho)^\alpha.
\end{equation}
\end{lemma}

\begin{proof}
Recalling Lemma \ref{l:rotation}, it suffices to show \eqref{e:A'} for the function
$\bar f_1$.
Let $n_0$ be the biggest integer such that $2^{n_0+3}\rho\leq \frac{1}{2}$ and
for every $k\leq n_0-1$ set $r_k=2^k\,\rho$.
If $\pi_k$ is such that $\E(T, B_{8 r_k}, \pi_k)=\E(T, B_{8 r_k})$,
then, by Corollary~\ref{c:decay_everywhere} (b), $|\pi-\pi_k|\leq C\,r_k^{1-\delta}$.
Hence, we conclude that the admissibility condition \eqref{e:good} holds with $r=r_k$,
so that we can consider the approximation $\bar f_k$ as in ($I_2$) for $r_k$.
From Corollary \ref{c:alpha}, we get
\begin{equation}\label{e:stima_alpha2}
\|D^\ell \bar f_k-D^\ell \bar{f}_{k+1}\|_{C^0 (B^m_{3r_k/2} (q))}
\leq C r_k^{3+\alpha-\ell}\leq
C\, 2^{-(n_0-k)(3+\alpha-\ell)} \quad\text{for }\;\ell\in \{0,1,2,3,4\}.
\end{equation}
Note that the series $\sum_i 2^{-i(3+\alpha-\ell)}$ is summable for $\ell\leq 3$.
Therefore, $\|\bar f_1\|_{C^3}\leq C+\|\bar f_{n_0}\|_{C^3}$.
On the other hand, since $r_{n_0}>1/32$, it is easy to see that $\|\bar f_{n_0}\|_{C^3}\leq C$
for some universal constant $C$, so that $\|\bar f_1\|_{C^3}\leq C$.
In the same way we have $\|D^4 \bar f_1\|_{C^0}\leq C\, \rho^{\alpha-1}$.
Then, \eqref{e:A'} follows from Lemma \ref{l:rotation} (c) .

Finally, Corollary \ref{c:alpha} obviously implies that
\begin{equation}\label{e:ancoraL^1}
\int_{B^m_{3r_k/2} (q)} |\bar f_k - \bar f_{k+1}| \leq C r_k^{m+3+\alpha}.
\end{equation}
Hence, using again Lemma \ref{l:rotation}, we conclude
\begin{equation}\label{e:ancoraL^1_2}
\int_{B^m_{r_k} (q')} |g_k - g_{k+1}| \leq C r_k^{m+3+\alpha}.
\end{equation}
Let $P_k$ and $P_{k+1}$ be the third order Taylor polynomials at $q'$ of $g_k$ and $g_{k+1}$.
From the estimate $\|D^4 g_k\|, \|D^4 g_{k+1}\|\leq C r_k^{\alpha-1}$
and \eqref{e:ancoraL^1_2}, we easily infer
\begin{equation*}%\label{e:ancoraL^1_3}
\int_{B^m_{r_k} (q')} |P_k - P_{k+1}| \leq C\, r_k^{m+3+\alpha}.
\end{equation*}
Hence, applying Lemma \ref{l:poly}, we then get
\begin{equation}\label{e:iterateD^3}
|D^3 g_k (q')-D^3 g_{k+1} (q')| =
|D^3 P_k (q')-D^3 P_{k+1} (q')| \leq C\, r_k^\alpha.
\end{equation}
Arguing as above, the estimate \eqref{e:stima_beta}
follows from \eqref{e:iterateD^3} and a simple iteration.
\end{proof}

The final step in the proof of Proposition~\ref{p:main} consists in comparing two different
interpolating functions defined at the same scale but for nearby balls and varying planes
$\pi$.
We do this in the following two separate lemmas.

\begin{lemma}\label{l:gamma}
Let $g_1$ and $g_2$ be the $(p,\rho, \pi)$- and $(p, \rho, \pi')$-interpolating functions
where as usual $(p,8\rho,\pi)$ and $(p,8\rho, \pi')$ are admissible.
Then,
\begin{equation}\label{e:gamma}
\sum_{\ell=0}^3\rho^{\ell-3-\alpha}\|D^\ell g_1-D^\ell g_2\|_{C^0 (B^m_\rho (q))} \leq C.
\end{equation}
\end{lemma}

\begin{proof}
As before, we first show that
\begin{equation}\label{e:g1g2}
\|g_1-g_2\|_{L^1 (B_{3/2 \rho} (q))}\leq C \rho^{m+3+\alpha}.
\end{equation}
Denote by $f_1, f_2$ the Lipschitz approximations given by Proposition~\ref{p:approx}
in the coordinates associated to $\pi, \pi'$ and
let $h_1, h_2: B_\rho (q)\to \pi_0^\perp$ be the Lipschitz functions whose graphs coincide
with the graphs of ${f}_1$ and $f_2$ respectively.
From Lemma~\ref{l:rotation} and Proposition~\ref{p:L1}, we have
\begin{equation*}%\label{e:compare_h_i}
\|g_i-h_i\|_{L^1 (B_{3/2 \rho} (q))}\leq 
\|f_i-\bar{f}_i\|_{L^1 (B_{2\rho} (q_i))}\leq C \,\rho^{m+3+\alpha},
\end{equation*}
where $(q_1, u_1)$, $(q_2, u_2)$ and $(q,u)$ are the coordinates of $p$ in
$\pi\times \pi^\perp$, $\pi'\times \pi'^\perp$ and $\pi_0\times \pi_0^\perp$ respectively.
Therefore, for \eqref{e:g1g2} it is enough to show
\[
\|h_1-h_2\|_{L^1 (B_{3/2 \rho} (q))}\leq C \rho^{m+3+\alpha}.
\]
To see this, consider the set $A=\{h_1\neq h_2\}$.
From Proposition~\ref{p:approx} if follows that
\begin{equation*}%\label{e:compare_graphs}
|A|\leq \mathcal{H}^m \big(\gr (h_1)\,\triangle\, \gr (h_2)\big)\leq C \rho^{m+2+\alpha}.
\end{equation*}
Then, if $x\in A$ and $y\in B_{3\rho/2}\setminus A$,
since $h_1 (y)=h_2 (y)$ and $\Lip (h_i)\leq C$, we have
\[
|h_1 (x) - h_2 (x)|\leq |h_1 (x)-h_1 (y)|+|h_2 (y) - h_2 (x)| \leq C|y-x|\leq C \rho,
\]
from which $\|h_1-h_2\|_{L^1 (B_{3/2 \rho} (q))}\leq C\,r\,|A|\leq C\, \rho^{m+3+\alpha}$.

From \eqref{e:g1g2} we are ready to conclude.
Let $x\in B_\rho (q)$ and $P_i$ be the third order Taylor expansions of $g_i$ at $x$.
Arguing as in Lemma \ref{l:beta}, we conclude
\[
\|P_1-P_2\|_{L^1 (B_{\rho/2} (x))}\leq C \rho^{m+3+\alpha}.
\]
Using Lemma \ref{l:poly} we then conclude
\begin{equation}\label{e:scaled_est}
|D^k P_1 (x) - D^k P_2 (x)|\leq C \rho^{3-k+\alpha} 
\quad \text{for }\;k\in \{0,1,2,3\}.
\end{equation}
On the other hand, since $D^k P_i (x)= D^k g_i (x)$,
\eqref{e:scaled_est} implies the desired estimates.
\end{proof}

\begin{lemma}\label{l:delta}
Let $g_1$ and $g_2$ be, respectively, the $(p,\rho, \pi)$- and $(p', \rho, \pi)$-interpolating
functions, where $(p,\rho,\pi)$ and $(p',\rho,\pi)$ are admissible.
Assume that $p=(q,u)$, $p'=(q',u')$ with $|q-q'|\leq \rho/16$.
Then,
\begin{equation}\label{e:delta}
\sum_{\ell=1}^4 \rho^{\ell-3-\alpha}\|D^\ell g_1-D^\ell g_2\|_{C^0 (B^n_\rho (q)\cap B^n_\rho (q'))} \leq C.
\end{equation}
\end{lemma}
The proof of this lemma exploits only a portion of the same computations used for Lemma~\ref{l:gamma}
and is left to the reader.

The proof of \eqref{e:B} follows straightforwardly from Lemma~\ref{l:gamma} and Lemma~\ref{l:delta};
while the proof of \eqref{e:C} is given below.
\begin{proof}[Proof of \eqref{e:C}]
Consider $R:= 16\,|q-q'|$ and let $h$, $k$ and $h'$ be the $(q, R, \pi)$-,
$(q, R, \pi')$- and $(q', R, \pi')$-interpolations, respectively.
By Corollary \ref{c:decay_everywhere}, if $|q-q'|$ is small enough,
we can apply Lemma \ref{l:gamma} and Lemma \ref{l:delta} to conclude that
$$
|D^3 h (q) - D^3 k (q)| + |D^3 k (q') - D^3 h' (q')|\leq C R^\alpha.
$$
On the other hand, by \eqref{e:A}, $\|D^4 k\|\leq C R^{\alpha-1}$,
and so $|D^3 h (q) - D^3 h' (q')|\leq R^\alpha$.
Since by Lemma \ref{l:beta} we know that $|D^3 g (q) - D^3 h (q)|\leq C
R^\alpha$
and $|D^3 g' (q') - D^3 h' (q')|\leq C R^\alpha$, the desired conclusion
follows.
\end{proof}

\appendix
\section{De Giorgi's regularity result}\label{s:appendix}
In this section we provide a proof of De Giorgi's regularity Theorem~\ref{t:degiorgi} in
its more refined version of Corollary~\ref{c:decay_everywhere}.
The overall strategy proposed here is essentially De Giorgi's celebrated original
one \cite{DG}; however, in many points we get advantage from
some new observations contained in our recent work \cite{DLSp2}.

In the following we keep the conventions of the rest of the paper,
but we use the various Greek letters $\alpha$, $\beta, \ldots$ for other
parameters and other functions.
Moreover, given a current $T$ in $\R{m+n}$, a Borel set $A\subset\R{m+n}$ and a simple $m$-vector $\tau$, we define the following excess measures:
\[
e(T,A,\tau):= \frac{1}{2}\int_A \vert \vec T-\tau\vert^2\,d\|T\|
\quad\text{and}\quad
e(T,A):= \min_{\|\tau\|=1} e(T,A,\tau).
\]

\subsection{Lipschitz approximation}
In this section we prove Proposition~\ref{p:approx}.
To this aim, we assume without loss of generality that $\pi=\pi_0$ and consider an area-minimizing integer rectifiable $m$-dimensional current $T$ in $\cyl_{r}$ such that
\[
\de \, T=0,\quad 
\proj_\#(T)=\a{B^m_r}\quad\text{and}\quad
E := \E (T, \cyl_{r})< 1,
\]
where $\proj: \R{m+n}\to \R{m}$ is the orthogonal projection.

The proof is in the spirit of the approximation result in \cite{DLSp2} and is made in three
steps.

\subsubsection{BV estimate}
Consider the push-forwards into the vertical direction of the $0$-slices
$\la T,\proj,x\ra$ through the projection $\proj^\perp:\R{m+n}\to \R{n}$:
\[
T_x:=\proj^\perp_\# \la T,\proj,x\ra.
\]
These integer $0$-currents (i.e.~sums of Dirac deltas with integer coefficients)
are characterized by the following identity (see \cite[Section 28]{Sim}):
\[
\int_{B_r}\la T_x,\psi\ra\ph(x)\, dx = \la T, \ph(x)\,\psi(y)\,dx\ra
\quad\text{for every }\,\ph\in C_c^\infty(B_r^m),\;\psi\in C_c^\infty(\R{n}).
\]
\begin{lemma}\label{l:A BV}
For every $\psi\in C_c^\infty(\R{n})$ with $\norm{D\psi}{\infty}\leq 1$, the
function $\Phi_\psi$ defined by $\Phi_\psi(x):=\la T_x,\psi\ra$ belongs to $BV (B_r^m)$ and
\begin{equation}\label{e:BV}
\big(|D\Phi_\psi|(A)\big)^2\leq 2\,e(T,A\times\R{n},\vec e_m)\,\|T\|(A\times\R{n})
\quad\text{for every open }\, A\subset B_r^m.
\end{equation}
\end{lemma}

\begin{proof}
For $\ph\in C^\infty_c(A,\R{m})$, note that 
$(\dv\, \ph(x))\,dx=d\alpha$,
where
\[
\alpha=\sum_j \ph_j\,d\hat x^j\quad\text{and }\quad
d\hat x^j=(-1)^{j-1}dx^1\wedge\cdots\wedge dx^{j-1}\wedge dx^{j+1}\wedge\cdots\wedge dx^{m}.
\]
Hence, from the characterization of the slices $T_x$, it follows that
\begin{align}\label{e:var}
\int_{A}\Phi_\psi(x)\,\dv \ph(x)\,dx &=
\int_{B_r} \la T_x,\psi(y)\ra \dv\ph(x)\,dx=
\la T,\psi(y)\,\dv\,\ph(x)\,dx\ra\notag\\
&=\la T, \psi\,d\alpha\ra
=\la T, d(\psi\,\alpha)\ra-\la T, d\psi\wedge\alpha\ra=-\la T, d\psi\wedge\alpha\ra,
\end{align}
where in the last equality we used the hypothesis $\de T=\emptyset$ on $B_r^m\times \R{n}$.
Now, observe that the 
$m$-form $d\psi\wedge\alpha$ has no $dx$ component, since
\[
d\psi\wedge\alpha=\sum_{j=1}^m\sum_{i=1}^n
\frac{\de \psi}{dy^i}(y)\,\ph_j(x)\,dy^i\wedge d\hat x^j.
\]
Write $\vec{T}=(\vec T\cdot \vec e_m)\,\vec{e}_m+\vec{S}$
(see \cite[Section 25]{Sim} for the scalar product on $m$-vectors).
We then conclude that
\begin{equation}\label{e:nohoriz}
\la T, d\psi\wedge\alpha\ra=\langle \vec{S}\cdot\norm{T}{}, d\psi\wedge\alpha\rangle.
\end{equation}
Moreover,
\begin{align}\label{e:JS L2}
\int_{A\times\R{n}}|\vec{S}|^2\,d\norm{T}{}&=
\int_{A\times\R{n}}\l(1-\big(\vec T\cdot \vec e_m\big)^2\r)\,d\norm{T}{}
\leq 2 \int_{A\times\R{n}}\l(1-\big(\vec T\cdot \vec e_m\big)\r)\,d\norm{T}{}\notag\\
&=2\,e(T,A\times\R{n},\vec e_m).
\end{align}
If $\norm{\ph}{\infty}\leq 1$, then $|d\psi\wedge\alpha|\leq
\norm{D\psi}{\infty}\,\norm{\ph}{\infty}\leq 1$.
Hence, by Cauchy--Schwartz inequality
\begin{align*}
\int_{A}\Phi_\psi(x)\,\dv\, 
\ph(x)\,dx &\stackrel{\eqref{e:var}}{\leq}|\la T, d\psi\wedge\alpha\ra|
\stackrel{\eqref{e:nohoriz}}{=}|\langle \vec{S}\cdot\norm{T}{}, d\psi\wedge\alpha\rangle|\leq
|d\psi\wedge\alpha| \int_{A\times\R{n}}|\vec{S}|\,d\norm{T}{}\\
&\leq \l(\int_{A\times\R{n}}|\vec{S}|^2\,d\norm{T}{}\r)^{\frac{1}{2}}\sqrt{\mass(T\res(A\times\R{n}))}\\
&\stackrel{\eqref{e:JS L2}}{\leq}\sqrt{2\,e(T,A\times\R{n},\vec e_m)}\,\sqrt{\mass(T\res(A\times\R{n}))}.
\end{align*}
Taking the supremum over all $\ph$ with $L^\infty$-norm less or equal $1$, we conclude \eqref{e:BV}.
\end{proof}

\subsubsection{Maximal Function truncation}
Here, we show how we determine the Lipschitz approximation.
Given $\alpha>0$, we set
\begin{gather*}
M_T (x) := \sup_{B_s^m (x)\subset B_r^m} \E (T,\cyl_s (x)),\\
K := \big\{x\in B_r^m:M_T(x)\leq E^{2\,\alpha}\big\}\quad\text{and}\quad
L:=\left\{x\in
B_r^m: M_T(x)>E^{2\,\alpha}/2^m\right\}.
\end{gather*}

\begin{lemma}\label{l:LipA}
Let $0<\alpha<\frac{1}{2}$ and $r'= r(1-E^{\frac{1-2\alpha}{m}})$.
Then, there exists $h:B_{r'}\to\R{n}$ such that:
\begin{gather}
\Lip(h)\leq C\, E^\alpha
\quad\text{and}\quad
\gr(h\vert_K)=T\res (K\times \R{n}),\label{e:lip1A}\\
|B_{s}\setminus K|\leq \frac{5^m}{E^{2\,\alpha}}\,e (T, (L\cap B_{s+ r E^{\frac{1-2\alpha}{m}}})\times \R{n}, \vec{e}_m)
\leq 5^m\,E^{1-2\,\alpha}\,r^m\quad\;\forall s\leq r'.\label{e:bad}
\end{gather}
\end{lemma}

\begin{proof}
Note that $x\notin K$ if and only if there exists $0<r_x<r\,E^{\frac{1-2\alpha}{m}}$ such that
\[
E^{2\,\alpha}<\frac{e(T,\cyl_{r_x}(x),\vec e_m)}{\omega_m\,r_x^m}
\leq \frac{e(T,\cyl_{r},\vec e_m)}{\omega_m\,r_x^m}
= \frac{r^m\,E}{r_x^m}.
\]
Hence, recalling the standard Maximal Function estimate (see, for example, \cite{St}),
we deduce easily \eqref{e:bad}.

In order to define the approximation $h$,
recall that
$\int_A \|T_x\|\leq \|T\|(A\times \R{n})$ for every open set A (cp.~to \cite[Lemma 28.5]{Sim}).
Therefore,
\[
\|T_x\|\leq\lim_{r\to 0}\frac{\|T\|(\cyl_r(x))}{\omega_m\,r^m}\leq M_T(x)+1 \quad\text{for almost every }x.
\]
Hence, since $E<1$ and $\proj_\sharp T = \a{B_r}$, we have that $1\leq\|T_x\|<2$ for almost every
point in $K$.
Thus, $T_x=\delta_{g(x)}$ for some measurable function $g$.

By Lemma~\ref{l:A BV}, for every $\psi\in C^\infty_c(\R{n})$ with $\|D\psi\|_{L^\infty}\leq1$,
\begin{align*}
M(|D\Phi_\psi|)(x)^2&=\sup_{0<s\leq r-|x|}\l(\frac{|D\Phi_\psi|(B_s (x))}
{|B_s|}\r)^2\leq
\sup_{0<s\leq r-|x|}\frac{2\,e(T,\cyl_s(x),\vec e_m)\,\mass(T,\cyl_s(x))}{|B_s|^2}\n\\
&= \sup_{0<s\leq r-|x|}\frac{2\,e(T,\cyl_s(x),\vec e_m)\big(e(T,\cyl_s(x),\vec e_m)+|B_s|\big)}
{|B_s|^2}\n\\
&\leq 2\,M_T(x)^2+2\,M_T(x)\leq C\, M_T(x).
\end{align*}
Therefore, by a standard argument (see, for instance, \cite[6.6.2]{EG}),
this implies the existence of 
a constant $C>0$ such that, for every $x,y\in K$ Lebesgue points of $\Phi_\psi$,
\[
|\Phi_\psi(x)-\Phi_\psi(y)|=\l|\psi(g(x))-\psi(g(y))\r|\leq C\,E^{\alpha}\,|x-y|.
\]
Taking the supremum over a dense, countable set of $\psi\in C_c^\infty(\R{n})$ with $\norm{D\psi}{\infty}\leq 1$, we deduce that
\begin{equation}\label{e:lip g}
|g(x) - g(y)|\leq C\,E^{\alpha}\,|x-y|.
\end{equation}
We can hence extend $g$ to all $B_{r'}$, obtaining a Lipschitz function $h$ with Lipschitz bound $C E^\alpha$.
Clearly, since $h\vert_K=g\vert_K$ and $T_x=\delta_{g(x)}$,
we conclude
$\gr(h\vert_K)=T\res (K\times \R{n})$.
\end{proof}

\begin{remark}
Note that from Lemma~\ref{l:LipA} it follows that
\begin{equation}\label{e:energy}
\int_{B_{r'}}|Dh|^2\leq C\,E \,r^m
\end{equation}
\begin{equation}\label{e:mass_of_diff}
\mbox{and}\quad \|T - \gr (h)\| (\cyl_{r'}) \leq C E^{1-2\alpha}\, r^m\, ,
\end{equation}
for some dimensional constant $C>0$
\end{remark}

\subsubsection{Proof of Proposition~\ref{p:approx}}
We start fixing positive constants $\alpha, \sigma, \theta, \gamma$ such that
\begin{equation}\label{e:constants}
\sigma,\gamma<\frac{1-2\,\alpha}{2m}\, ,\quad 2\,\sigma < \theta<\gamma
\quad\text{and}\quad \frac{1-2\,\alpha-\sigma}{m-1}\,m>1.
\end{equation}
Consider the Lipschitz approximation $h$ given by Lemma~\ref{l:LipA} corresponding to the exponent $\alpha$ (we keep the same notation as above).
By a slicing argument, we find $s\in [r(1-E^\sigma),r(1-E^\theta)]$ such that
\begin{equation}\label{e:theta?}
\int_{B_{s+ r E^\theta}\setminus B_{s-2 r E^\theta}}|Dh|^2\leq C\, E^{\sigma}\int_{B_{r'}}|Dh|^2\leq C\,E^{1+\sigma}\,r^m\,
\end{equation}
\begin{equation}\label{e:extraslice}
\mbox{and}\qquad
\mass ((T-\gr (h))\res \partial B_s)\leq C\, E^{1-2\alpha -\sigma}r^{m-1}\, .
\end{equation}
(With a slight abuse of notation, we write
$(T-\gr(h))\res \partial B_s$ for
$\la T-\gr(h), \ph,s\ra$, where $\ph (x) = |x|$.)

Moreover, setting for a standard kernel $\ph$
\begin{equation*}
g(x)=
\begin{cases}
h*\ph_{rE^\gamma} & \text{if } x\in B_{s-rE^\theta},\\
\frac{|x|-s+r\,E^\theta}{r\,E^\theta}\,h(x)+\frac{s-|x|}{r\,E^\theta}\,h*\ph_{rE^\gamma}(x) & \text{if } x\in B_s\setminus B_{s-rE^\theta},
\end{cases}
\end{equation*}
it is simple to verify that $\Lip(g)\leq CE^\alpha$ and, furthermore,
\begin{equation}\label{e:key}
\int_{B_s}|Dg|^2\leq \int_{B_s\setminus L}|Dh|^2+C\,E^{1+\delta}\,r^m,
\end{equation}
for some $\delta>0$, where $L$ is as in Lemma~\ref{l:LipA}.
Indeed, we can estimate the energy of $g$ in two steps as follows.
First in the annulus $B_s\setminus B_{s-rE^\theta}$:
\begin{align*}%\label{e:est1}
&\int_{B_s \setminus B_{s-rE^\theta}}|Dg|^2\leq C\,\int_{B_{s+ r E^\gamma}\setminus B_{s-r(E^\theta+E^\gamma)}}|Dh|^2+
\frac{C}{r^2\,E^{2\,\theta}}\int_{B_s\setminus B_{s-rE^\theta}}|h-h*\ph_{rE^\gamma}|^2\notag\\
&\leq C\,\int_{B_{s+rE^\theta}\setminus B_{s-2rE^\theta}}|Dh|^2+
\frac{C\,r^2\,E^{2\,\gamma}}{r^2\,E^{2\,\theta}}\int_{B_r}|Dh|^2
\stackrel{\eqref{e:energy}}{\leq} C\,(E^{1+\sigma}+E^{1+2\,\gamma-2\,\theta})\,r^m\, .
\end{align*}
Hence, in $B_{s-rE^\theta}$:
\begin{align*}
\int_{B_{s-rE^\theta}}|Dg|^2&\leq \int_{B_{s-rE^\theta}}(|Dh|*\ph_{rE^\gamma})^2
=\int_{B_{s-rE^\theta}}\big((|Dh|\,{\bf 1}_{B_s\setminus L}+|Dh|\,{\bf 1}_{B_s\cap L})*\ph_{rE^\gamma}\big)^2\\
&\leq 2\int_{B_{s-rE^\theta}}((|Dh|\,{\bf 1}_{B_s\setminus L})*\ph_{rE^\gamma})^2+
2\int_{B_{s-rE^\theta}}((|Dh|\,{\bf 1}_{B_s\cap L})*\ph_{rE^\gamma})^2,
% \notag\\
% &\leq 2\int_{B_{s}}(|Dh|\,{\bf 1}_{B_s\cap L})^2+
% 2\,\|Dh\|_{L^\infty}\,\|{\bf 1}_{B_s\cap L}*\ph_{rE^\gamma}\|^2_{L^2}\\
% &\leq 2\int_{B_{s}\setminus L}|Dh|^2+2\int_{B_s\setminus B_{s-rE^\theta}}|Dh|^2+
% C E^{2\,\alpha}\,\|{\bf 1}_{B_s\cap L}*\ph_{rE^\gamma}\|^2_{L^2},
\end{align*}
where the first term is estimated in turn as
\[
\int_{B_{s-rE^\theta}}((|Dh|\,{\bf 1}_{B_s\setminus L})*\ph_{rE^\gamma})^2
\leq \int_{B_{s}}(|Dh|\,{\bf 1}_{B_s\setminus L})^2
\leq \int_{B_{s}\setminus L}|Dh|^2\,
\]
and the second one as follows,
\begin{align*}
\int_{B_{s-rE^\theta}}(|Dh|\,{\bf 1}_{B_s\cap L}*\ph_{rE^\gamma})^2&\leq
\|Dh\|_{L^\infty}\,\|{\bf 1}_{B_s\cap L}*\ph_{rE^\gamma}\|^2_{L^2}
\leq C\,E^{2\,\alpha}\,\|\ph_{rE^\gamma}\|^2_{L^2}\,\|{\bf 1}_{B_s\cap L}\|^2_{L^1}\\
&=
C\,E^{2-m\,\gamma-2\,\alpha}\,r^m.
\end{align*}
Hence, by the choice of the constants in \eqref{e:constants}, \eqref{e:key} follows.

Next, we observe that, from $\de \big(T-\gr(h)\big)\res \de B_r=0$,
by the isoperimetric inequality and \eqref{e:extraslice},
there is an integer rectifiable current $R$ such that
$$
\de R=\big(T-\gr(h)\big)\res \de B_s
\quad\text{and}\quad
\|R\|\leq CE^{\frac{(1-2\alpha-\sigma)m}{m-1}}\,r^m.
$$
Moreover, being $g\vert_{\de B_s}=h\vert_{\de B_s}$,
we can use $\gr(g)+R$ as competitor for the current $T$. In this way we obtain,
for a suitable $\tau>0$,
\begin{equation}\label{e:massa1}
\|T\|( \cyl_s)\leq 
|B_s|+\int_{B_s}\frac{|Dg|^2}{2}+C\,E^{1+\tau}\, r^m\stackrel{\eqref{e:key}}
{\leq}|B_s|+
\int_{B_s\setminus L}\frac{|Dh|^2}{2}+C\,E^{1+\tau}\,r^m.
\end{equation}
On the other hand, again using the Taylor expansion for the area functional,
\begin{align}\label{e:massa2}
\|T\|(\cyl_s)&=\|T\|((B_s\cap L)\times \R{n})+
\|\gr(h\vert_{B_s\setminus L})\|\n\\
&\geq \|T\|((B_s\cap L)\times \R{n})+ |B_s\setminus L|+
\int_{B_s\setminus L}\frac{|Dh|^2}{2}-C\, E^{1+\tau}\,r^m.
\end{align}
Hence, from \eqref{e:massa1} and \eqref{e:massa2}, we deduce
\begin{equation}\label{e:walm}
e(T,(B_s\cap L)\times\R{n},\vec e_m)\leq C\,E^{1+\tau}\,r^m.
\end{equation}

We are now in the position to conclude the proof of Proposition~\ref{p:approx}.
Let $\beta<\alpha$ be such that $2\beta<\tau$ and let $f$ be the Lipschitz approximation 
given by Lemma~\ref{l:LipA} corresponding to $\beta$.
Clearly, \eqref{e:approx1} follows once we take $\eta\leq\beta$.
Moreover, since $\{M_T>E^\beta/2^m\}\subset L$, from \eqref{e:bad} and \eqref{e:walm} we get \eqref{e:approx2} if $\eta$ is accordingly chosen.
Finally, for \eqref{e:approx3}, we use again the Taylor expansion of the area functional to conclude:
\begin{align*}%\label{e:iii}
\left| \|T\|(\cyl_s) - \omega_m \,s^m-
\int_{B_s} \frac{|Df|^2}{2}\right|
\leq{}& e(T,(B_s\cap L)\times\R{n},\vec e_m)
+\int_{B_s\cap L} \frac{|Df|^2}{2} + C\, E^{1+\beta}\, r^m\\
\leq& C\, (E^{1+\tau}+E^{1+\beta})\,r^m+C\, E^{2\beta}|B_1\cap L|\leq C\,E^{1+\eta}\,r^m.\notag
\end{align*}

\subsection{Convergence to harmonic functions}\label{s:harm}
Let $(T_l)_{l\in\N}$ be a sequence of minimizing $m$-currents in $B_1\subset\R{m+n}$
such that
\begin{equation}\label{e:H-contro}
\partial T_l = 0\; \text{in}\; B_1,
\quad \theta (T_l, 0)=1
\quad\text{and}\quad \|T_l\| (B_1)\leq \omega_m+\eps_l\quad\text{with}\;\eps_l\to0.
\end{equation}
%with $\eps_l\to0$.
It is immediate to see that, up to subsequences, the $T_l$ converge in the sense of currents to a flat $m$-dimensional disk centered at the origin.
By the monotonicity formula, there is also Hausdorff convergence of the supports
of $T_l$ to the flat disk
in every compact set $C\subset\subset B_1$:
\[
\lim_{l\to+\infty} \sup_{x\in C\cap\supp(T_l)} \dist(x,\R{m}\times\{0\})=0.
\]
In particular, there exist radii $r_l\to 1$ such that
\[
\de \big(T_l\res \cyl_{r_l}\big)=0\quad \text{in}\;\; \cyl_{r_l}
\quad\text{and}\quad
\proj_\# \big(T_l\res \cyl_{r_l}\big)=\a{B_{r_l}^m}.
\]

In the following proposition we prove the convergence to a harmonic function for the
rescaled Lipschitz approximations.

\begin{propos}\label{p:harm}
Let $T_l$ be as in \eqref{e:H-contro}, $E_l:=e(T_l,\cyl_{r_l},\vec e_m)$ and $f_l: B_{r_l(1-E_l^\eta)}\to\R{n}$ the
approximations in Proposition~\ref{p:approx}.
The rescaled functions $u_l:=\frac{f_l-\bar f_l}{\sqrt{E_l}}$,
where $\bar f_l=\fint f_l$ are the averages, converge in $W^{1,2}_{loc}$ to a
harmonic function $u$.
\end{propos}

\begin{proof}
Note that, by \eqref{e:energy} it follows that $\sup_l \int_{B_{r_l}}|Du_l|^2 <\infty$.
Hence, since $\fint u_l=0$, by the
Sobolev embedding and the Poincar\'e inequality,
there exists a function $u:B_1\to\R{n}$ such that,
for every $s<1$, $u_l\to u$ in $L^2(B_s)$ and $Du_l\weak Du$ in $L^2(B_s)$.

Set $D_{s}=\liminf_{l}\int_{B_s}|Du_l|^2$.
If the proposition does not hold, for some $s<1$, then
\begin{itemize}
\item[(i)] either $\int_{B_s}|Du|^2<D_s$,
\item[(ii)] or $u|_{B_s}$ is not harmonic.
\end{itemize}
Under this assumption, we can find $s_0>0$ such that, for every $s\geq s_0$,
there exists $v\in W^{1,2}(B_{s},\R{n})$ with
\begin{equation}\label{e:non minimal}
v\vert_{\de B_{s}}=u\vert_{\de B_{s}}
\quad\text{and}\quad
\gamma_s:= D_s - \int_{B_s}|Dv|^2>0.
\end{equation}

With a slight abuse of notation, we write
$(T_l-\gr(f_l))\res \partial \cyl_r$ for
$\la T_l-\gr(f_l), \ph,r\ra$, where $\ph (z,y) = |z|$. Consider
the function $\psi_l$ given by
\[
\psi_l(r) := E_l^{-1}\mass \big((T_l-\gr(f_l))\res \partial \cyl_r\big)
+ \int_{\de B_r}|Du_l|^2 + \int_{\de B_r}|Du|^2 
+\frac{\int_{\de B_r} |u_l- u|^2}{\int_{B_{s_l}}|u_l-u|^2},
\]
Since from the estimates on the Lipschitz approximation, one gets
\begin{align*}%\label{e:diff mass}
\|T_l\|(\cyl_{s_l})-\|\gr(f_l)\|(\cyl_{s_l}) \leq C E_l^{1+\eta},
\end{align*}
$\liminf_l \int_{s_0}^{s_l} \psi_l (r)\, dr < \infty$.
Therefore, by Fatou's Lemma, there is $s\in (s_0,1)$ and a subsequence, not relabelled,
such that $\lim_l \psi_l (s) < \infty$. It follows that:
\begin{itemize}
\item[(a)] $\int_{\de B_s}|u_l-u|^2\to 0$,
\item[(b)] $\int_{\de B_s}|Du_l|^2 + \int_{\de B_s}|Du|^2\leq M$ for some $M<\infty$,
\item[(c)] $\|(T_l-\gr(u_l))\res\de B_s\|\leq C\,E_l$.
\end{itemize}
Once fixed $s$, we approximate $v$ by a Lipschitz function $w$ such that:
\[
\int_{B_s}|Dw|^2\leq \int_{B_s}|Dv|^2+\theta,\quad
\int_{\de B_s}|Dw|^2\leq \int_{\de B_s}|Dv|^2+\theta\quad\text{and}\quad
\int_{\de B_s}|w-v|^2 \leq \theta,
\]
where $\theta>0$ will be chosen later.
Next, for every given $\delta>0$ (also to be chosen later),
define $\xi_l$ via a linear interpolation so that
\[
\Lip(\xi_l)\leq CE_l^{\eta-1/2},
\quad\xi_l|_{\de B_s}=u_l|_{\de B_s} \quad \xi_l|_{B_{s-\delta}} = w|_{B_{s-\delta}}\, .
\]
It is easy to see that this can be done so that the following estimates hold:
\begin{align*}
\int_{B_s}|D\xi_l|^2
&\leq \int_{B_s}|Dw|^2+C \delta\,\int_{\de B_s}|Dw|^2+C \delta\,\int_{\de B_s}|Du_l|^2
+C\,\delta^{-1}\,\int_{\de B_{s}}|w-u_l|^2\\
&\stackrel{(a),(b)}{\leq} \int_{B_s} |Dv|^2 + \theta + C \,\delta\, M + C \,\delta^{-1} \theta .
\end{align*}
We choose first $\delta$ and then $\theta$ so to guarantee that
\begin{equation}\label{e:guadagno finale}
\limsup_l \int |D\xi_l|^2 \;\leq\; \int_{B_s} |Dv|^2 + \frac{\gamma_s}{2}\, .
\end{equation}

The functions $\xi_l$ give the desired contradiction.
Set $z_l:=\sqrt{E_l}\, \xi_l$ and consider the current 
$Z_l := \gr (\xi_l)$. Since $z_l|_{\de B_s} = f_l|_{\de B_s}$,
$\partial Z_l = \gr (f_l) \res \de B_s$. Therefore, from (c),
$\|\partial (T_l\res B_s - Z_l)\|\leq C E_l$.
From the isoperimetric inequality (see \cite[Theorem~30.1]{Sim}),
there exists an integral current $R_l$ such that
\[\partial R_l= \partial (T_l\res \cyl_s - Z_l)\quad\text{and}\quad 
\|R_l\| \leq C E_l^{m/(m-1)}.
\]
Set finally $W_l = T_l \res (\cyl_{r_l}\setminus \cyl_s) + Z_l + R_l$. By construction, it holds
$\partial W_l = \partial T_l$.
The Taylor expansion of the area functional and the various
estimates achieved give:
\begin{align*}
\limsup_l \frac{\|W_l\| - \|T_l\|}{E_l}
& \leq \limsup_l \frac{\|W_l\| - \|\gr(f_l)\|+C\,E_l^{1+\eta}}{E_l}\n\\ 
&\leq
\limsup_l E_l^{-1} \left\{\|R_l\| + \int_{B_{s}}\frac{|Dz_l|^2}{2} - \int_{B_{r}}\frac{|Df_l|^2}{2}\right\}\n\\
&\leq \limsup_l \int_{B_s} \frac{|D\xi_l|^2}{2} - D_s \;\leq\; \int_{B_s} |Dv|^2 + \frac{\gamma_s}{2} - D_s
\;\leq\; - \frac{\gamma_s}{2} < 0\, .
\end{align*}
For $l$ large enough this last inequality contradicts the minimality of the current
$T_l$.
\end{proof}

\begin{remark}\label{r:energy}
Note the following easy corollary of Proposition~\ref{p:approx}:
\[
\int_{B_s}|Df|^2=2\,E\,\omega_m\,r^m+C\,E^{1+\eta}\,r^m.
\]
Hence, in particular, the harmonic function $u$ in Proposition~\ref{p:harm} satisfies
\[
\int_{B_1}|Du|^2\leq 2\,\omega_m.
\]
\end{remark}

\subsection{Decay estimate}
We start with the following technical lemma.

\begin{lemma}\label{l:taylor}
Let $f:B_s\to\R{n}$ be the Lipschitz approximation in
Proposition~\ref{p:approx}
and $S$ the integer current associated to its graph.
If $\tau$ is the unitary $m$-vector given by
\begin{equation}\label{e:tau}
\tau
=\frac{(e_1+A\,e_1)\wedge\cdots\wedge(e_m+A\,e_m)}{\|(e_1+A\,
e_1)\wedge\cdots\wedge(e_m+A\,e_m)\|},\quad\text{with }\; A=\fint_{B_s}Df,
\end{equation}
then, for every $t\leq s$,
\[
e(S,\cyl_t,\tau)=\int_{B_t}\frac{|Df-A|^2}{2}+CE^{1+\eta}.
\]
\end{lemma}

\begin{proof}
From the definition of excess and \eqref{e:approx2}, it follows that
\begin{equation}\label{e:ex1}
e(S,\cyl_t,\tau)=\|S\|(\cyl_t)-\int_{\cyl_t}\langle \vec S,\tau\rangle\,d\|S\|
=\int_{B_t}\frac{|Df|^2}{2}+|B_t|+C\,E^{1+\eta}-\int_{\cyl_t}\langle \vec
S,\tau\rangle\,d\|S\|.
\end{equation}
Notice that $|A|^2=|\fint Df|^2\leq \fint |Df|^2\leq CE$, thus implying
\[
|\bar\tau|=\sqrt{\la \bar \tau,\bar \tau\ra}=\sqrt{\det(\delta_{ij}+Ae_i\cdot
Ae_j)}=
\sqrt{1+|A|^2+O(|A|^4)}
=1+\frac{|A|^2}{2}+O(E^2)
\]
(where $\bar\tau = (e_1+A\, e_1)\wedge\ldots \wedge (e_m+A\, e_m)$)
and
\begin{multline*}
\langle (e_1+Df\,e_1)\wedge\cdots\wedge(e_m+Df\,e_m), \bar{\tau} \rangle=\\
=\det(\delta_{ij}+Df e_i\cdot Ae_j)=
1+Df\cdot A+O(|Df|^2|A|^2)
=1+Df\cdot A+O(E^{1+2\eta}).
\end{multline*}
Hence, from
\[
\vec S=
\frac{(e_1+Df\,e_1)\wedge\cdots\wedge(e_m+Df\,e_m)}{\|(e_1+Df\,e_1)\wedge\cdots\wedge(e_m+Df\,e_m)\|},
\]
we have that
\begin{align}\label{e:ex2}
\int_{\cyl_t}\langle \vec S,\tau\rangle\,d\|S\|&=
\int_{B_t}(1+Df\cdot A+O(E^{1+2\,\eta}))\,\left(1+\frac{|A|^2}{2}+O(E^2)\right)^{-1}\,dx\notag\\
&=\int_{B_t}\left(1+Df\cdot A-\frac{|A|^2}{2}\right)\,dx+O(E^{1+2\,\eta}).
\end{align}
Putting together \eqref{e:ex1} and \eqref{e:ex2}, we obtain the desired conclusion,
\[
e(S,\cyl_t,\tau)=\int_{B_t}\frac{|Df|^2+|A|^2-2\,Df\cdot
A}{2}+O(E^{1+\eta})=
\int_{B_t}\frac{|Df-A|^2}{2}+O(E^{1+\eta}).
\]
\end{proof}

The basic step in De Giorgi's decay estimate is the following.

\begin{propos}\label{p:basic iter}
For every $\theta>0$, there exists $\eps>0$ such that, if $T$ is an area-minimizing $m$-dimensional integer rectifiable current in $B_1$ such that
\[
\partial T = 0,
\quad \theta (T, 0)=1
\quad\text{and}\quad \|T\| (B_1)\leq \omega_m +\eps,
\]
then
\begin{equation}\label{e:basic iter}
e(T,B_{\frac{1}{2}})\leq \left(\frac{1}{2^{m+2}}+\theta\right)\,e(T,B_1).
\end{equation}
\end{propos}
\begin{proof}
The proof is by contradiction. Assume that there exists $\theta>0$ and a sequence of
area-minimizing currents $(T_l)_{l\in\N}$ in $B_1$ satisfying \eqref{e:H-contro} such that
\begin{equation}\label{e:contra}
e(T_l,B_{\frac{1}{2}})> \left(\frac{1}{2^{m+2}}+\theta\right)\,e(T_l,B_1).
\end{equation}
Let $r_l$ be the sequence in Section~\ref{s:harm}. Clearly, it holds that
\begin{equation*}
e(T_l,\cyl_{r_l},\vec e_m)= e(T_l,\cyl_{r_l}\cap B_1,\vec e_m)\leq e(T_l,B_1,\vec
e_m)=e(T_l,B_1).
\end{equation*}
Hence, from \eqref{e:contra},
\begin{align}\label{e:contra2}
e(T_l,\cyl_{\frac{1}{2}},\vec e_m)&\geq e(T_l,\cyl_{\frac{1}{2}})\geq
e(T_l,B_{\frac{1}{2}})>\left(\frac{1}{2^{m+2}}+\theta\right)\,e(T_l,B_1)\notag\\
&\geq \left(\frac{1}{2^{m+2}}+\theta\right)\,e(T_l,\cyl_{r_l},\vec e_m).
\end{align}
Let $E_l,f_l, u_l$ be as in Proposition~\ref{p:harm} and $u:B_1\to\R{n}$ the harmonic
function such that $u_l$ converges to $u$ in $W^{1,2}_{loc}(B_1)$.
Note that \eqref{e:contra2} and Remark~\ref{r:energy} imply that
\[
\int_{B_{\frac{1}{2}}}|Du|^2=\lim_{l\to+\infty}\int_{B_{\frac{1}{2}}}|Du_l|^2
=\lim_{l\to+\infty}\frac{2\,e(T_l,\cyl_{\frac{1}{2}})}{E_l}>
2\,\left(\frac{1}{2^{m+2}}+\theta\right)>0.
\]
In particular, $Du$ is not identically $0$.
Since from Remark~\ref{r:energy} $\int_{B_1}|Du|^2\leq 2\omega_m$, from \eqref{e:contra2} and Lemma~\ref{l:taylor}, we get
\begin{align*}
\int_{B_{\frac{1}{2}}}\frac{|Df_l-A_l|^2}{2}+C\,E_l^{1+\eta}\geq
e(T_l,\cyl_{\frac{1}{2}},\tau_l)\geq e(T_l,\cyl_{\frac{1}{2}})
\geq \left(\frac{1}{2^{m+2}}+\theta\right)\,E_l\,\int_{B_1}\frac{|Du|^2}{2}.
\end{align*}
where $A_l=\fint_{B_{1/2}}Df_l$ and $\tau_l$ is as in \eqref{e:tau}.

Rescaling by $E_l$ and passing to the limit in $l$, for $(Du)_s=\fint_{B_s}Du$,
we get
\begin{equation*}
\int_{B_{\frac{1}{2}}}|Du-(Du)_\frac{1}{2}|^2\geq
\left(\frac{1}{2^{m+2}}+\theta\right)\,\int_{B_1}|Du|^2\geq
\left(\frac{1}{2^{m+2}}+\theta\right)\,\int_{B_1}|Du-(Du)_1|^2,
\end{equation*}
against the decay property of harmonic functions. This gives the contradiction and concludes the proof.
\end{proof}

We conclude with the proofs of Proposition~\ref{p:degiorgi_improved}
and Corollary~\ref{c:decay_everywhere}.

\begin{proof}[Proof of Proposition~\ref{p:degiorgi_improved}]
The proof is now an easy consequence of Proposition~\ref{p:basic iter}.
For every $\delta>0$ choose $\theta>0$ such that
\[
\left(\frac{1}{2}\right)^{2-2\,\delta}=\frac{1}{4}+2^m\,\theta.
\]
Fix $\bar{\eps}$ sufficiently small so that
Proposition \ref{p:basic iter} applies.
Then, if we chose $\eps$ small enough in the hypothesis (H) of Proposition~\ref{p:degiorgi_improved}, recalling
the beginning of Section \ref{s:harm}, we have that
\[
\|T\| (B_{1/2} (p))\leq (\omega_m +\bar{\eps})\,2^{-m}\quad
\text{for every }\,p\in \supp (T)\cap B_{1/2}.
\]
It follows from the monotonicity formula that
$\|T\| (B_r (p))\leq (\omega_m +\bar{\eps})r^m$ for
every $r$.
Therefore, we can apply iteratively (the appropriate
rescaled version of) Proposition~\ref{p:basic iter}, and for $2^{-k-1}<r\leq 2^{-k}$ one obtains
\[
\E(T,B_r (p))\leq C \E(T,B_{\frac{1}{2^k}} (p))\leq C\left(\frac{1}{2^k}\right)^{2-2\,\delta}\E(T,B_{1/2} (p))
\leq C\, \E (T, B_1) \,r^{2-2\,\delta}\, .
\]
\end{proof}

\begin{proof}[Proof of Corollary~\ref{c:decay_everywhere}]
Consider two admissible pairs $(p,\rho,\pi)$ and $(p,2\rho,\pi')$.
Using the monotonicity $\|T\|(B_s(p))\geq \omega_m s^m$ for ever $s>0$,
it follows that
\[
|\vec{\pi} - \vec{\pi}'|^2\leq C\,\rho^{-m}\int_{B_\rho}|\vec{\pi}-\vec T|^2\,d\|T\|
+C\,(2\rho)^{-m}\int_{B_{2\rho}}|\vec{\pi}'-\vec T|^2\,d\|T\|\leq C\,\eps^2_0\,\rho^{2-2\,\delta}.
\]
Hence, for admissible pairs $(p,\rho,\pi)$ and $(p,\rho',\pi')$ with
$\frac{\rho}{2^{k+1}}<\rho'\leq \frac{\rho}{2^k}$, we have
\begin{equation}\label{e:exist tang2}
|\vec{\pi} - \vec{\pi}'|^2\leq C\,\eps^2_0\,\sum_{i=0}^k\left(\frac{\rho}{2^k}\right)^{2-2\,\delta}
\leq C\,\eps^2_0\,\rho^{2-2\,\delta}\sum_{i=0}^k 2^{-(2-2\,\delta)k}\leq C\,\eps^2_0\,\rho^{2-2\,\delta}.
\end{equation}
As already noticed, Proposition~\ref{p:degiorgi_improved} implies that there exists an admissible $\vec \pi_{\rho,p}$ for every $p\in B_{\frac{1}{2}}\cap \supp(T)$ and $\rho\leq\frac{1}{8}$.
Therefore, from \eqref{e:exist tang2} one deduces the existence of
the limit plane $\vec\pi_p=\lim_{\rho\to0}\vec\pi_{\rho,p}$.
Moreover, the same computations imply that, if
$\pi$ is admissible, then $|\vec\pi-\vec\pi_p|\leq C'\,\eps_0\,\rho^{1-\delta}$.
Vice versa, if $|\vec\pi-\vec\pi_p|\leq C''\,\eps_0\,\rho^{1-\delta}$,
then $\pi$ is admissible (the constant $C'$ is possibly larger than $C''$).
Finally, from $\E(T,B_\rho,\vec\pi_p)\leq
C\,\eps_0^2\,\rho^{2-2\,\delta}$,
it follows straightforwardly that $\pi_p$ is the tangent plane to $T$ at $p$,
thus proving (b).

Reasoning as above, for $p,p'\in\supp(T)\cap B_{\frac{1}{2}}$, setting $r=|p-p'|$, we have that
\begin{align*}%\label{e:pi cont}
|\vec{\pi}_p - \vec{\pi}_{p'}|^2&\leq
\frac{C}{r^m}\int_{B_r(p)\cap B_r(p')}\big(|\vec\pi_p-\vec T|^2+|\vec\pi_{p'}-\vec T|^2\big)\notag\\
&\leq
\frac{C}{r^m}\int_{B_r(p)}|\vec\pi_p-\vec T|^2+\frac{C}{r^m}\int_{B_r(p')}|\vec\pi_{p'}-\vec T|^2\leq C\,\eps^2_0\,r^{2-2\,\delta}.
\end{align*}
It follows that,
if $(p,\rho,\pi)$ and $(p,\rho',\pi')$ are admissible and, to fix
the ideas $\rho'\leq\rho$, then
\begin{equation}\label{e:quasi holder}
|\vec{\pi} - \vec{\pi}'|^2\leq C\,\rho^{2-2\,\delta}+|\vec{\pi}_p - \vec{\pi}'_{p'}|^2
\leq C\,\eps^2_0\,\max\{\rho,\rho',|p-p'|\}^{2-2\,\delta}.
\end{equation}
This is all we need to conclude. Indeed, by the fact
that $\proj_\#(T\res B_{1/2}\cap \cyl_{1/4})=\a{B_{1/4}}$,
it follows that $\supp(T)$ is a graph of a function on $B_{1/4}$, thus giving
(c).
Moreover, from $|\vec\pi_0-\vec\pi_p|^2\leq C\,\eps^2_0\,|p|^{2-2\delta}\leq
C\,\eps^2_0$,
it follows that the Lipschitz constant of this function is bounded by $C\,\eps_0$.
Hence, $|p-p'|\leq C|q-q'|$ for $p,p'\in\supp(T)\cap B_{1/2}\cap\cyl_{1/4}$, where $q=\proj(p)$, $q'=\proj(p')$,
and estimate (a) follows from \eqref{e:quasi holder}.
\end{proof}

% %\nocite{*}
%\bibliographystyle{plain}
%\bibliography{reference}

\end{document}